\documentclass[a4paper]{article}

\usepackage[all]{xy}\usepackage[latin1]{inputenc}        %accents
\usepackage[dvips]{graphics,graphicx}
\usepackage{amsfonts,amssymb,amsmath,color,mathrsfs, amstext,bm}
\usepackage{amsbsy, amsopn, amscd, amsxtra, amsthm,authblk, enumerate}
\usepackage{upref}
\usepackage[colorlinks,
            linkcolor=red,
            anchorcolor=red,
            citecolor=red
            ]{hyperref}

\usepackage{geometry}
\usepackage[displaymath]{lineno}
\usepackage{float}

\usepackage{yhmath}

\usepackage{enumitem}

\usepackage[normalem]{ulem}

\usepackage{xcolor}

\numberwithin{equation}{section}

\def\R{\mathbb{R}}

\usepackage{graphicx}
\usepackage{subcaption}

\usepackage{algorithm}
\usepackage{algpseudocode}
\usepackage{caption}
\newtheorem{theorem}{Theorem}[section]
\newtheorem{lemma}{Lemma}[section]
\newtheorem{assumption}{Assumption}[section]
\newtheorem{proposition}{Proposition}[section]
\newtheorem{remark}{Remark}[section]
\newtheorem{corollary}{Corollary}[section]

\newcommand\keywords[1]{\textbf{Keywords:} #1}

\DeclareMathOperator*{\argmin}{argmin}

\newcommand{\tone}{\textbf{1}}

\newtheorem*{theorem*}{Theorem}

\usepackage{todonotes}

\begin{document}
\title{Relative entropy estimate and geometric ergodicity for implicit Langevin Monte Carlo
%On Implicit Langevin Monte Carlo: \\ Relative Entropy Error Estimates and Ergodicity
}

\author[a]{Lei Li\thanks{E-mail:leili2010@sjtu.edu.cn}}
\author[b,c]{Jian-Guo Liu\thanks{E-mail:jliu@math.duke.edu}}
\author[b]{Yuliang Wang\thanks{E-mail:yuliang.wang2@duke.edu}}
\affil[a]{School of Mathematical Sciences, Institute of Natural Sciences, MOE-LSC, Shanghai Jiao Tong University, Shanghai, 200240, P.R.China.}
\affil[b]{Department of Mathematics,  Duke University, Durham, NC 27708, USA.}
\affil[c]{Department of Physics, Duke University, Durham, NC 27708, USA.}
%\affil[d]{Shanghai Artificial Intelligence Laboratory}

\date{}
\maketitle

\begin{abstract}
We study the implicit Langevin Monte Carlo (iLMC) method, which simulates the overdamped Langevin equation via an implicit iteration rule. In many applications, iLMC is favored over other explicit schemes such as the (explicit) Langevin Monte Carlo (LMC). LMC may blow up when the drift field $\nabla U$ is not globally Lipschitz, while iLMC has convergence guarantee when the drift is only one-sided Lipschitz. Starting from an adapted continuous-time interpolation, we prove a time-discretization error bound under the relative entropy (or the Kullback-Leibler divergence), where a crucial gradient estimate for the logarithm numerical density is obtained via a sequence of PDE techniques, including Bernstein method. Based on a reflection-type continuous-discrete coupling method, we prove the geometric ergodicity of iLMC under the Wasserstein-1 distance. Moreover, we extend the error bound to a uniform-in-time one by combining the relative entropy error bound and the ergodicity. Our proof technique is universal and can be applied to other implicit or splitting schemes for simulating stochastic differential equations with non-Lipschitz drifts.
\end{abstract}



% 65C30  	Numerical solutions to stochastic differential and integral equations {For theoretical aspects, see 60H35} [See also 65M75, 65N75]

% 60H10  	Stochastic ordinary differential equations (aspects of stochastic analysis) [See also 34F05]

% 82M31  	Monte Carlo methods applied to problems in statistical mechanics [See also 65C05]

\section{Introduction}

Effective simulation of a stochastic differential equation (SDE) is crucial in many real-world applications, including generative diffusion models, high-dimensional Bayesian inference, molecular dynamics, finance, etc \cite{higham2011stochastic, hodgkinson2021implicit, chewi2024analysis, deng2020semi}. We are in particular interested in simulating SDEs whose drifts may grow super-linearly. It is known that in the case where the drift is not Lipschitz, explicit schemes such as the (forward) Euler-Maruyama scheme may blow up \cite{mattingly2002ergodicity,hutzenthaler2011strong}, while implicit schemes tend to be more stable and have convergence guarantee when the drift field is only one-sided Lipschitz \cite{hu1996semi, higham2002strong, liu2023backward}. 
% and in this case, regular schemes such as (forward) Euler-Maruyama scheme tends to fail. 
In this paper, we study the \textbf{implicit Langevin Monte Carlo} (iLMC) method, which simulates overdamped Langevin equation via an implicit iteration rule. The iLMC can also be viewed as an effective high-dimensional sampling algorithm, since the overdamped Langevin equation with a potential function $U$ has an invariant measure $\pi \propto e^{-U}$, which can be viewed as the target distribution in many practical sampling tasks. In applications, iLMC is favored over other explicit schemes (for instance, the Langevin Monte Carlo (LMC), which is the explicit Euler's scheme for the overdamped Langevin equation), especially when the tail of the target distribution behaves like $e^{-|x|^p}$ ($p > 2$).
% This method maintains desirable properties such as ergodicity and enables better control over sampling errors, especially in terms of Wasserstein distances. 

Let us first explain the iLMC iteration. Given a potential $U: \mathbb{R}^d \rightarrow \mathbb{R}$ and a standard $d$-dimensional Brownian motion $(W_t)_{t\geq 0}$ on a probability space $(\Omega, \mathcal{F}, \mathbb{P})$ with the natural filtration $(\mathcal{F}_t)_{t\geq 0}$, the overdamped Langevin equation is given by:
% For the target distribution $\pi \propto e^{-U}$, we have the following overdamped Langevin equation whose invariant measure is $\pi$:
\begin{equation}\label{eq:overdamped}
    dX = -\nabla U(X) dt + \sqrt{2}\, dW.
\end{equation}
% Here $U: \mathbb{R}^d \rightarrow \mathbb{R}$ is the potential function, and $W$ is the standard Brownian motion in $\mathbb{R}^d$. It is well-known that under suitable conditions (citation) \eqref{eq:overdamped} converges to its invariant measure $\pi$ as the time goes to infinity. Hence, to generate approximated samples from $\pi$, one needs to seek effective numerical schemes that simulate the continuous-time equation \eqref{eq:overdamped}. 
Given a constant step size $h$, denote $t_n := nh$ for $n = 0,1,2,\dots$. The iLMC, or equivalently the backward Euler's discretization of \eqref{eq:overdamped} is then given by
\begin{equation}\label{eq:ilmciteration}
    X^h_{t_{n+1}} = X^h_{t_n} - h \nabla U(X^h_{t_{n+1}}) + \sqrt{2} \,\Delta W_n,
\end{equation}
where $\Delta W_n = W_{t_{n+1}} - W_{t_n}$ is the Wiener increment and the implicit nature of the scheme comes from evaluating the gradient at $X^h_{t_{n+1}}$ instead of $X^h_{t_n}$. Note that we will assume in Assumption \ref{ass0} below that $\nabla^2 U$ is continuous and $U$ is strongly convex in the far field. Consequently, the mapping $x \mapsto x + h \nabla U(x)$ is reversible, and so the iLMC iteration \eqref{eq:ilmciteration} is always well-defined. Moreover, for small $h$, \eqref{eq:ilmciteration} is identical to the one-step iteration of the minimizing movement scheme \cite{de1993new, ambrosio1995minimizing}
\begin{equation}\label{eq:jko}
    X^h_{t_{n+1}} = \argmin_{x \in \mathbb{R}^d} \left\{U(x) + \frac{1}{2h}\left|x - (X^h_{t_n} + \sqrt{2}\,\Delta W_n)\right|^2 \right\},
\end{equation}
which is well-defined and has stability under Assumption \ref{ass0}. See more details in Proposition \ref{prop:PhiLip} below.

In literature, there exist plenty of results involving theoretical analysis for the above backward Euler's discretization, even in the presence of the Brownian motion. When the coefficients of SDEs are all Lipschitz, classical convergence theorems tell us that the backward Euler's scheme has first-order strong convergence and second-order weak convergence \cite{kp92}. In past decades, researchers have been more interested in the case of non-Lipschitz drift coefficients. To our knowledge, the earliest result on strong convergence of backward Euler's scheme is \cite{hu1996semi}. Mainly under the one-sided Lipschitz assumption for the drift, the author of \cite{hu1996semi} proved a first-order strong convergence of the form $\int_0^T \mathbb{E}|X^h_t - X_t|^2 dt$, where $X^h_t$ is the same continuous-time interpolation as we use in this paper (see \eqref{eq:interpolation} below).
In \cite{higham2002strong}, mainly assuming the one-sided Lipschitz and polynomial growth conditions for the drift function, the authors obtained a (finite-time) first-order strong convergence under a stronger metric $\mathbb{E}[\sup_{0\leq t\leq T} |\bar{X}^h_t - X_t|^2]$, where they used a different continuous-time interpolation $\bar{X}^h_t$ based on an intermediate split-step backward Euler scheme. The convergence analysis for more variants of backward Euler's scheme applied to various models \cite{higham2007strong, mao2013strong, he2015unconditional, zhang2019backward, deng2020semi} under various metrics, such as weak convergence \cite{wang2024weak}, $L^p$ strong convergence \cite{liu2022p}, convergence under Wasserstein distances \cite{liu2023backward}, etc. Remarkably, for the simulation of overdamped Langevin equation, the authors of \cite{hodgkinson2021implicit} studied the $\theta$-Euler's scheme (semi-implicit, semi-explicit), and a total-variation-based geometric ergodicity and a central-limit-type theorem were established.
However, to the best of our knowledge, existing results is limited to weaker metrics such as Wasserstein distances. In recent years, the relative entropy (or the more general R\'enyi divergence) has received increasing popularity when measuring the effectiveness of sampling algorithms including iLMC. Note that although not a true distance, the relative entropy can control other classical distances such as total variation and Wasserstein distances though some transportation inequalities \cite{otto2000generalization,talagrand1991new,bolley2005weighted,pinsker1963information}.
% However, to the best of our knowledge, few existing results take into account the error between distributions, which is often measured by relative entropy, total variation distance or Wasserstein distances. This is in fact much more important compared with strong errors, since as a discretization of the Langevin equations, we are actually analyzing a sampling algorithms, where the distance between the numerical densities and the target distribution is of vital importance.

Motivated by this, we provide a novel approach to study the relative entropy error of iLMC. Below, we briefly summarize the main contributions of this paper.
% The main contribution of our work is a novel proof framework for both error analysis and ergodicity of the iLMC method, under a non-globally-Lipschitz and non-globally-dissipative setting for the drift $\nabla U$ (see more details in Assumptions \ref{ass0} - \ref{ass2} below). In particular, for the first time, we give a time-discretization error bound under the relative entropy (or the Kullback-Leibler divergence), which is a stronger result compared with some other error bounds under total variation or Wasserstein distances.  
First, in Sections \ref{sec:error} and \ref{sec:nablalogrho}, starting with a continuous-time interpolation and explicitly expressing the iLMC with an adapted It\^o's process, we prove a second-order error bound in terms of the relative entropy. 
% Second, in Section\ref{sec:ergodicity}, by constructing a reflection-type continuous-discrete coupling, we rigorously prove the geometric ergodicity of the iLMC iteration under the Wasserstein-1 distance.
% Moreover, combining the former two results, in Section \ref{sec:longtime}, we extend the error estimate to a uniform-in-time one in terms of the Wasserstein-1 distance.\tcb{11}
In detail, we consider the interpolation
\begin{equation}\label{eq:interpolation}
    X^h_s = X^h_{t_n} - (s - t_n)\nabla U(X^h_s) + \sqrt{2}(W_s - W_{t_n}),\quad s \in [t_n, t_{n+1}).
\end{equation}
Note that stochastic processes defined in \eqref{eq:overdamped}, \eqref{eq:ilmciteration}, \eqref{eq:interpolation} are driven by the same Brownian motion. However, our analysis would not be influenced if one chooses different Brownian motions to define these processes, since in both our results and proofs, we only focus on the law of these processes, namely, the behaviors of Fokker-Planck equations rather than SDEs.
Clearly, under Assumption \ref{ass0} below, the process $X^h_s$ is always well-defined, and the solution coincides with \eqref{eq:ilmciteration} at the time grids $t_n$ ($n = 0,1,2,\dots$). It is also obvious that the process $X^h_s$ is adapted. Now, although the continuous-time interpolation \eqref{eq:interpolation} is of an implicit form, we can in fact rewrite it via It\^o's calculus (see the explicit formula in \eqref{eq:explicitSDE} below). Then we analyze the relative entropy error based on an explicit Fokker-Planck equation describing the time evolution of the law of $X^h_s$. Based on Assumptions \ref{ass0} -- \ref{ass2} below, we prove the following relative entropy error bound (see Theorem \ref{thm:errormain} for a complete statement):

\begin{theorem*}
Fix $T>0$. Denote $\rho^h_s$, $\rho_s$ the laws of $X^h_s$, $X_s$, respectively. Then for small time step $h$ one has
\begin{equation}\label{eq:relativeentropyerror}
    \sup_{s \in [0,T]}\mathcal{H}\left(\rho^h_s \mid \rho_s \right) \leq Ch^2.
\end{equation}
Here, $C$ is a positive constant that may depend algebraically on $T$, and $\mathcal{H}$ denotes the relative entropy.
\end{theorem*}
% \begin{equation}\label{eq:relativeentropyerror}
%     \sup_{s \in [0,T]}\mathcal{H}\left(\rho^h_s \mid \rho_s \right) \leq C(T)h^2,
% \end{equation}
% where $\rho^h_s$, $\rho_s$ are laws of $X^h_s$, $X_s$, respectively, and $\mathcal{H}$ denotes the relative entropy. 
Notably, the rigorous derivation for \eqref{eq:relativeentropyerror} also requires one to obtain a pointwise polynomial upper bound for $\nabla \log \rho^h$.
% involves some nontrivial estimates for $\rho^h_s$ and $\nabla \log \rho^h_s$. 
The non-Lipschitz  drift in the current settings makes the derivation more challenging compared with known results. We resolve this using a sequence of PDE techniques, including Bernstein method for gradient estimate \cite{bernstein1906generalisation,bernstein1910generalisation,li1986parabolic,du2024collision,ju2025modified, feng2023quantitative}. 
In fact,  applying Bernstein method, we are able
to obtain 
$$|\nabla u(x)| \leq \mathcal{P}(x)(1+|u(x)|)$$
where $u = \log ( \rho^h / M_0)$ ($M_0 > 0$) solves an Hamilton-Jacobi equation after Cole-Hopf transformation of $\rho^h$, and $\mathcal{P}(x)$ is a polynomial. Further, we show that $|u|$ itself has a polynomial upper bound by studying the tail behaviors of $\rho^h$. This then gives a polynomial upper-bound for $|\nabla \log \rho^h|$. 
See more details in Section \ref{sec:nablalogrho} below.

Another contribution of this work is a novel proof of the geometric ergodicity of iLMC (as a discrete-time Markov chain) in terms of the Wasserstein-1 distance. Under the far-field confining condition (Assumption \ref{ass0} below), we show that: (see the complete statement in Theorem \ref{thm:contraction} below)
\begin{theorem*}
Let $\mu_n$, $\nu_n$ be laws of iLMC solution $X^h_{t_n}$ with different initial distributions $\mu_0$, $\nu_0$. Then for small time step $h$, there exist positive constants $C_0$, $C$ independent of $h$ and $n$ such that
\begin{equation}\label{eq:W1contraction}
    W_1(\mu_n, \nu_n) \leq C_0e^{-Cnh} W_1(\mu_0, \nu_0).
\end{equation}
\end{theorem*}
A direct consequence of the Wasserstein contraction result \eqref{eq:W1contraction} is that:  iLMC as a discrete-time Markov chain has a unique invariant measure $\pi^h$, and the law of iLMC converges exponentially fast to $\pi^h$ under Wasserstein-1 distance.
In order to prove \eqref{eq:W1contraction}, similarly as in \cite{li2025ergodicity}, we propose a reflection-type continuous-discrete coupling method. Intuitively, each step of the iLMC iteration \eqref{eq:ilmciteration} can be separated into two steps -- the (pure) diffusion step and the deterministic mapping step. The deterministic step can be shown to be stable, and contractive in the far-field region. For the diffusion step, we make use of the continuous-time reflection coupled Brownian motions and a concave, increasing Lyapunov function. 
Note that reflection coupling is a classical technique to study the contraction of It\^o's processes with drifts that are dissipative only in the far field. In recent decades, it has been applied to analyze various systems such as the overdamped Langevin equation \cite{eberle2016reflection, wang2020exponential, luo2016exponential}, the underdamped Langevin equation \cite{eberle2019couplings,schuh2024global}, the interacting particle systems \cite{eberle2016reflection, jin2023ergodicity}, (discrete-time) Langevin Monte Carlo \cite{li2024geometric, majka2020nonasymptotic}, to name a few. Combining the estimates for the two steps, we are able to prove \eqref{eq:W1contraction}. The detailed derivation is given in Section \ref{sec:ergodicity} below. 
% Consequently, iLMC as a discrete-time Markov chain is geometric ergodic, namely, $X^h_{t_n}$ ($n = 0,1,2,\dots$) has a unique invariant measure $\pi^h$, and its law converges exponentially to $\pi^h$ under Wasserstein-1 distance.

Finally, by combining the relative entropy error bound and the Wasserstein-1 contraction result, and using the semigroup property as well as the propagation of some basic properties of the Fokker-Planck equation associated with the overdamped Langevin equation \eqref{eq:overdamped}, we extend the error estimate into a uniform-in-time one under the Wasserstein-1 distance. See more details in Section \ref{sec:longtime} below.

The rest of this paper is organized as follows. In Section \ref{sec:interpolation}, after introducing the basic assumptions, we propose the continuous-time interpolation of the iLMC iteration and explicitly derive the corresponding Fokker-Planck equation in detail. We then prove the main result of the relative entropy error estimate for iLMC in Section \ref{sec:error}. One key gradient estimate via Bernstein method for logarithm numerical density is derived in Section \ref{sec:nablalogrho}, and in this paper this result is used in Section \ref{sec:error}. In Section \ref{sec:ergodicity}, we prove the geometric ergodicity of iLMC using a reflection type coupling technique. Combining the results obtained in Section \ref{sec:error} -- \ref{sec:ergodicity}, we prove an extended Wasserstein-1 error bound for iLMC that is valid uniformly in time. Section \ref{sec:conclusion} gives some conclusion and further discussions, and some technical lemmas are proved in the Appendix.

\section{Setup and a continuous-time interpolation}\label{sec:interpolation}

As mentioned in the section above, our proof framework for the relative entropy error estimate begins with an adapted continuous-time interpolate process. We still denote it by $X^h_s$, $s \in [0,T]$ for some fixed $T > 0$.
\begin{equation}\label{eq:interpolation2}
    X^h_s = X^h_{t_n} - (s - t_n)\nabla U(X^h_s) + \sqrt{2}\,(W_s - W_{t_n}),\quad s \in [t_n,t_{n+1}).
\end{equation}

Next, we derive an explicit formula for the process $X^h_s$ in the form of It\^o's integral.
Before the detailed derivations, let us begin with some basic assumptions for this paper. It is easy to verify that a super-linearly growing potential such as $U(x) = |x|^4 - |x|^2$ in Ginzburg-Landau model satisfies all the following assumptions.

% (The ergodicity requires this assumption only.)
\begin{assumption}\label{ass0}
The Hessian matrix $\nabla^2 U$ satisfies:
\begin{enumerate}
\item There exists $m>0$ and $R>0$ such that
\begin{equation*}
    \nabla^2 U(x) \succeq mI,\quad \forall |x| \geq R.
\end{equation*}
\item $\nabla^2 U$ is continuous on $\mathbb{R}^d$. Consequently, for any $r>0$, there exists $M(r) > 0$ such that $\max_{x \in B(0,r)} |\nabla^2 U(x)| = M(r) < \infty$. In particular, we denote
\begin{equation*}
    M := \max_{x \in B(0,R)} |\nabla^2 U(x)|.
\end{equation*}
\end{enumerate}
Without loss of generality, we also assume that $U(x) \geq 0$ for all $x \in \mathbb{R}^d$, since $U$ is bounded from below under these two conditions above.
\end{assumption}
Note that we are not requiring a global convexity condition for the potential $U$, which is far too restrictive for many applications. Moreover, the above far-field confining condition is enough to derive the geometric ergodicity result for iLMC proved in Section \ref{sec:ergodicity} below.

In order to derive the relative entropy error bound in Section \ref{sec:error}, we require the following conditions for the potential $U$ and the initial distribution $\rho_0$.

\begin{assumption}\label{ass1}
The potential $U$ satisfies $U \in C^5(\mathbb{R}^d)$, and there exists $C>0$, $\ell \geq 1$ such that for $k = 1,2,3,4,5$,
\begin{equation*}
    |\nabla^k U(x)| \leq C(1 + |x|^\ell),\quad \forall x \in\mathbb{R}^d.
\end{equation*}
Moreover, for $k = 2,3,4$,
\begin{equation*}
    |\nabla^k U(x)| \leq C\left( |\nabla^{k-1} U(x)| + 1\right),\quad \forall x \in\mathbb{R}^d.
\end{equation*}
% and 
% \begin{equation*}
%     \Delta U(x) \leq C(1 + x \cdot \nabla U(x)),\quad \forall x \in\mathbb{R}^d.
% \end{equation*}
\end{assumption}

Here all $| \cdot|$ means the operator norm, i.e. $|\nabla^k U| = \sup \{\sum_{1 \leq i_1, i_2,\dots, i_k \leq d }\partial_{i_1 i_2\dots i_k} U \, v_{i_1} v_{i_2} \dots v_{i_k}: |v| = 1, v\in\mathbb{R}^d  \}$. In particular, it is Euclidean norm when $k=1$ and matrix 2-norm when $k=2$.

\begin{assumption}\label{ass2}
The numerical scheme (iLMC) $X^h$ and the true solution $X$ share the same initial distribution $\rho_0$, and for all $p \geq 1$, the $p$-th moment $\int_{\mathbb{R}^d} |x|^p \rho_0(dx)$ is finite. Moreover, there exist $C_0, C_1, C_2, C_3 >0$, $\gamma \in (0,1)$ and $\ell_1 \geq 3\ell + 2$ such that 
\begin{equation*}
    |\nabla \log \rho_0(x)| \leq C_0(1+|x|^{\ell_1}),\quad C_1\exp(-C_2 |x|^{\ell_1}) \leq \rho_0(x) \leq C_3\exp(-\gamma U(x)),\quad \forall x \in \mathbb{R}^d.
\end{equation*}
\end{assumption}

% \begin{example}
% A simple concrete example of $U$ and $\rho_0$ that satisfies Assumptions \ref{ass0}, \ref{ass1}, \ref{ass2} is: $U(x) = |x|^4$ and $\rho_0$ being the standard Gaussian distribution. Note that the tail behavior $\rho_0 \leq C_3 \exp(-\gamma U)$ holds because $U$ can be lower-bounded by a quadratic function in the far field under Assumption \ref{ass0}.
% \end{example}

% \begin{assumption}
% $\nabla^k U$ ($k=1,2,3$ and $k=4,5$ for bounding $\nabla \log \rho^h$) polynomial [grow at far-field?], $\nabla^2 U$ locally bounded, $\nabla^2 U$ strongly convex in the far-field regiom. (the ergodicity only requires the later two)
% \end{assumption}

% \begin{assumption}
% finite moment of the initial; same initial distribution of numerical and true solution; $|\nabla \log \rho_0| \leq C(1+|x|^\ell)$  [grow at far-field?].
% ; $C_1\exp(-C_2 |x|^\ell) \leq \rho_0(x) \leq C_3.$

% \end{assumption}

A direct consequence of the assumptions above is the following lemma, which will be frequently used in the subsequent analysis of this paper.

\begin{lemma}\label{lmm:matrixbound}
Suppose Assumption \ref{ass0} holds and recall the definition of $m$, $M$, $R$  therein. For all $h \in (0,1/(2M))$ and $x \in \mathbb{R}^d$, the matrix
\begin{equation*}
    I +  h \nabla^2 U(x)
\end{equation*}
is invertible, and 
\begin{equation}
    \left|\left(I + h  \nabla^2 U(x)\right)^{-1}\right| \leq 
    \left\{
    \begin{aligned}
        & e^{-\frac{1}{2}mh},\quad |x| \geq R,\\
        & e^{2Mh},\quad |x| < R.
    \end{aligned}
    \right.
\end{equation}
% [Explain $m$, $M$][explained in the assumptions above]
\end{lemma}

\begin{proof}
For $x \in \mathbb{R}^d$, let $\underline{\lambda}(x) \in \mathbb{R}$ be the smallest eigenvalue of $\nabla^2 U$ (note that under the current assumption, the Hessian matrix $\nabla^2 U$ is symmetric so it only has real-valued eigenvalues). Recall $M = \sup_{|x| \leq R}|\nabla^2 U(x)|$. Clearly, $0 \leq M < \infty$ under the current assumption (without loss of generality we assume $M > 0$ throughout our analysis). So we have $\underline{\lambda}(x)$ lower bounded by $-M$ when $|x| \leq R$, and by $m$ when $|x| > R$. Consequently, when $h < 1 / M$, the matrix $I + h \nabla^2 U(x)$ is always invertible, and
\begin{equation*}
    \left|\left(I + h  \nabla^2 U(x)\right)^{-1}\right| \leq 
    \left\{
    \begin{aligned}
        & (1 + m h)^{-1},\quad |x| \geq R,\\
        & (1 - Mh)^{-1},\quad |x| < R.
    \end{aligned}
    \right.
\end{equation*}
Moreover, for $h < 1 / (2M)$, $(1 + mh)^{-1} \leq 1 - \frac{1}{2}mh \leq e^{-\frac{1}{2}mh}$ and $(1 - Mh)^{-1} \leq 1 + 2Mh \leq e^{2Mh}$.

\end{proof}

In order to see the well-definedness of iLMC \eqref{eq:ilmciteration} and its continuous-time interpolation \eqref{eq:interpolation2} more clearly, for $h>0$, we define the map $\Phi_h: \mathbb{R}^d \rightarrow \mathbb{R}^d$ by
\begin{equation}\label{eq:defphiearliest}
    \Phi_h(x) := x + h \nabla U(x).
\end{equation}
Then, once $\Phi^{-1}_h$ is well-defined for small $h$, we can rewrite \eqref{eq:ilmciteration} as
\begin{equation}\label{eq:ilmcrewrite}
    X_{t_{n+1}}^h = \Phi^{-1}_{h}\left(X^h_{t_n} + \sqrt{2}(W_{t_{n+1}} - W_{t_n}) \right),
\end{equation}
and \eqref{eq:interpolation2} as
\begin{equation}\label{eq:interpolationrewrite}
    X_s^h = \Phi^{-1}_{s-t_n}\left(X^h_{t_n} + \sqrt{2}(W_s - W_{t_n}) \right),\quad s \in[t_n,t_{n+1}).
\end{equation}
The introduction of the map $\Phi_h$ is also helpful during the proof of ergodicity in Section \ref{sec:ergodicity} below. We prove some crucial properties of $\Phi_h$ here. Note that the stability \eqref{eq:stability1} below also corresponds to the stability of the minimizing movement scheme \eqref{eq:jko}.

\begin{proposition}\label{prop:PhiLip}
Suppose Assumption \ref{ass0} holds with constants $m$, $M$ therein. Fix $h \in (0,1/(2M))$. Then $\Phi_h$ is a homeomorphism, and $x=\Phi_h^{-1}(x_0)$ is equivalent to
\begin{gather}\label{eq:jkoeuclidean}
x=\mathrm{argmin}_{x\in \R^d}\left\{U(x)+\frac{|x-x_0|^2}{2h}\right\}.
\end{gather}
Moreover, there exists $R' = (4 + 16M / m) R$ such that the inverse satisfies
    \begin{equation}\label{eq:lipphi}
        \left| \Phi^{-1}_h(x) - \Phi^{-1}_h(y)\right| \leq
        \left\{
        \begin{aligned}
            & e^{-\frac{m}{4}h} |x-y|,\quad |x-y| > R',\\
            & e^{2Mh} |x-y|, \quad |x-y| \leq R'.
        \end{aligned}
        \right.
    \end{equation}
    Consequently, $\Phi^{-1}_h$ has stability in the sense that $U(\Phi^{-1}_h(x))\le U(x)$ and
\begin{equation}\label{eq:stability1}
    |\Phi^{-1}_h(x)| \leq C \vee \left((1 - C'h)|x|\right),\quad\forall x \in \mathbb{R}^d.
\end{equation}
where $C$, $C'$ are independent of $h$.
\end{proposition}

\begin{proof}
We first verify the well-definedness of the inverse map. The fact that $\Phi_h$ is onto is clear by the existence of the minimizer in \eqref{eq:jkoeuclidean}, due to the fact that $U$ is convex outside a compact set. Then, it suffices to show that $\Phi_h$ is injective. In fact, suppose there exists $x_1, x_2 \in \mathbb{R}^d$ ($x_1 \neq x_2$) such that $\Phi_h(x_1) = \Phi_h(x_2)$. Recall that $\nabla \Phi_h$ is globally positive definite by Lemma \ref{lmm:matrixbound}. Consider the function $g: [0,1] \rightarrow \mathbb{R}$ defined by
$$g(\theta) := \Phi_h(x_1 + \theta (x_2 - x_1)) \cdot (x_2 - x_1).$$
Since
\begin{equation*}
    g'(\theta) = (x_2 - x_1) \cdot \nabla \Phi_h(x_1 + \theta (x_2 - x_1)) \cdot (x_2 - x_1) > 0,
\end{equation*}
one has $g(0) < g(1)$. This is a contradiction with $\Phi_h(x_1) = \Phi_h(x_2)$, which implies $g(0) = g(1)$. Hence, the equivalence to \eqref{eq:jkoeuclidean} is then clear.

Next, we prove the Lipschitz property \eqref{eq:lipphi}.
By definition, denoting $z_\lambda := \lambda x + (1-\lambda) y$ ($\lambda \in [0,1]$), one has
\begin{equation*}
    \left| \Phi^{-1}_h(x) - \Phi^{-1}_h(y)\right| = \left|\int_0^1 \left(I + h\nabla^2 U\left(\Phi_h^{-1}(z_\lambda) \right) \right)^{-1} d\lambda \cdot (x-y) \right|.
\end{equation*}
Clearly, under Assumption \ref{ass0}, for $h < 1/(2M)$, $\left|\left(I + h\nabla^2 U\left(\Phi_h^{-1}(z_\lambda) \right) \right)^{-1} \right|$ is bounded by $1 - \frac{1}{2}m h$ when $|z_\lambda| > R$, and by $1 + 2Mh$ when $|z_\lambda| \leq R$. Consequently,  for all $x, y \in \mathbb{R}^d$,
\begin{equation*}
    \left|\Phi^{-1}_h(x) - \Phi^{-1}_h(y)\right| \leq e^{2Mh} |x-y|.
\end{equation*}
Moreover, when $|x-y| > R' = (4 + 16M / m) R$, the largest length of $\{\lambda \in [0,1]: |z_\lambda| \leq R \}$ is $2R$. Hence, when $|x-y| > R'$,
\begin{multline*}
     \left|\Phi^{-1}_h(x) - \Phi^{-1}_h(y)\right| \leq \frac{2R}{R'}(1 + 2Mh) + \left(1 - \frac{2R}{R'} \right)\left(1 - \frac{1}{2}m h\right)\\
     = 1 - \left(\frac{1}{2}m -(m + 4M)\frac{R}{R'}\right)h = 1-\frac{1}{4}m h.
\end{multline*}
The fact $U(\Phi^{-1}_h(x))\le U(x)$ is a direct consequence of the optimization scheme \eqref{eq:jkoeuclidean}, and \eqref{eq:stability1} is a direct consequence of \eqref{eq:lipphi}.

\end{proof}

Now, let us come back to the continuous-time interpolation \eqref{eq:interpolation2}. We have the following proposition.
\begin{proposition}\label{prop:explicitSDE}
Suppose Assumption \ref{ass0} holds and recall the definition of $m$, $M$, $R$  therein. For $h < 1/(2M)$,
\begin{enumerate}
\item The iLMC iteration \eqref{eq:ilmciteration} and the continuous-time interpolation \eqref{eq:interpolation2} are well-defined, and they share the same value at time grids $t_n$ for $n = 0,1,2,\dots$. Moreover, the $p$th moment of $X^h_s$ defined in \eqref{eq:interpolation2} has uniform bounds. Namely, for any $p \geq 2$, if
\begin{equation*}
    \mathbb{E}|X^h_0|^p < \infty,
\end{equation*}
then there exists a positive constant $C_p$ independent of $h$ and $t$ such that 
\begin{equation}\label{eq:Lpmomentbound}
    \sup_{t\geq 0}\mathbb{E}|X^h_t|^p \leq C_p < \infty.
\end{equation}
\item \eqref{eq:interpolation2} is an It\^o's process with the following explicit expression:
\begin{equation}\label{eq:explicitSDE}
    dX^h_s = b_h(s,X^h_s) ds + \sqrt{2}\sqrt{\Lambda_h(s,X^h_s)}\, dW,
\end{equation}
where
\begin{multline}\label{eq:bn}
    b_h(s, x) :=  -\left(I + (s - t_n)\nabla^2 U(x)\right)^{-1} \nabla U(x) \\
          -  (s - t_n) \left(I + (s - t_n)\nabla^2 U(x)\right)^{-1} \left( \nabla^3 U(x) :  \left(I + (s - t_n)\nabla^2 U(x)\right)^{-2}\right),
\end{multline}
and 
\begin{equation}\label{eq:Lambdan}
    \Lambda_h(s,x) := \left(I + (s - t_n)\nabla^2 U(x)\right)^{-2}.
\end{equation}
Consequently, its law $\rho^h_s$ satisfies the following Fokker-Planck equation
\begin{equation}\label{eq:numericalFP}
    \partial_s \rho^h_s = -\nabla \cdot (b_h(s,x) \rho^h_s) + \nabla^2 : (\Lambda_h(s,x) \rho^h_s).
\end{equation}
\end{enumerate}
\end{proposition}

\begin{proof}
The first part in Claim 1 above is obvious due to \eqref{eq:ilmcrewrite}, \eqref{eq:interpolationrewrite} and Proposition \ref{prop:PhiLip} above.
% Indeed, define the map $\Phi_h: \mathbb{R}^d \rightarrow \mathbb{R}^d$ by
% \begin{equation}\label{eq:defphiearliest}
%     \Phi_h(x) := x + h \nabla U(x).
% \end{equation}
% It suffices to verify that $\Phi_h$ is injective. Suppose there exists $x_1, x_2 \in \mathbb{R}^d$ ($x_1 \neq x_2$) such that $\Phi_h(x_1) = \Phi_h(x_2)$. Since $\nabla \Phi_h$ is globally positive definite by Lemma \ref{lmm:matrixbound}, one knows the function $\theta \mapsto \Phi_h(x_1 + \theta (x_2 - x_1)) \cdot (x_2 - x_1)$ is strictly increasing in $[0,1]$. This means $\Phi_h(x_1) \neq \Phi_h(x_2)$, leading to a contradiction. Note that the above argument holds for all $h \in (0,h_0)$, and the definition of \eqref{eq:interpolation2} is equivelant to
% \begin{equation*}
%     X_s^h = \Phi^{-1}_{s-t_n}\left(X^h_{t_n} + \sqrt{2}(W_s - W_{t_n}) \right).
% \end{equation*}
% One immediately knows that $X_s^h$ is also well-defined for $s \in [t_n,t_{n+1})$. Also, it is obvious that \eqref{eq:ilmciteration} and  \eqref{eq:interpolation2} share the same value at all time grids.

The moment bound \eqref{eq:Lpmomentbound} is relatively standard in literature, though most of which are in the sense of $L^2$ instead of $L^p$ ($p\geq 2$) here (see for instance \cite{mao2013strong, liu2023backward}). Our proof relies on a stability property \eqref{eq:stability1}.
% In our case, although we are handling a continuous-time process, $X^h_t$ ($t \in [t_n, t_{n+1})$) is in fact the solution of one-step iLMC starting from $X^h_{t_n}$. Hence, the technique for the moment estimate is similar. 
We refer the reader to Appendix \ref{app:nablalogrho} for a complete proof of \eqref{eq:Lpmomentbound}. 

In what follows, we prove \eqref{eq:explicitSDE}.
Differentiating \eqref{eq:interpolation2} yields:
\[
    dX^h_s = -\nabla U(X^h_s)\, ds - (s - t_n)\left(\nabla^2 U(X^h_s) \cdot dX^h_s + \frac{1}{2} \nabla^3 U(X^h_s) : d[X^h_s, X^h_s] \right) + \sqrt{2}\, dW_s,
\]
where $d[X^h_s, X^h_s]$ is the quadratic variation and we will handle it later.
We can rewrite the expression more compactly as:
\begin{equation}\label{eq:dXearlier}
    \begin{aligned}
    dX^h_s &= -  (I + (s - t_n) \nabla^2 U(X^h_s)) ^{-1}\nabla U(X^h_s) \, ds \\
    &\quad -  \frac{s - t_n}{2}  (I + (s - t_n) \nabla^2 U(X^h_s))^{-1} \left( \nabla^3 U(X^h_s) : d[X^h_s, X^h_s] \right)\\
    &\quad + \sqrt{2}  (I + (s - t_n) \nabla^2 U(X^h_s))^{-1} dW_s.
    \end{aligned}
\end{equation}
Note that the martingale term is:
\[
    \sqrt{2} (I + (s - t_n) \nabla^2 U(X^h_s))^{-1} dW_s.
\]
Then the quadratic variation of $dX_s$ becomes:
\[
    d[X^h_s, X^h_s] = 2\left(I + (s - t_n)\nabla^2 U(X^h_s)\right)^{-2} ds.
\]
Substituting this expression for $d[X^h_s, X^h_s]$ back into \eqref{eq:dXearlier}, we obtain:
\[
\begin{aligned}
    dX^h_s &= -\left(I + (s - t_n)\nabla^2 U(X^h_s)\right)^{-1} \nabla U(X^h_s) \, ds \\
         &\quad -  (s - t_n) \left(I + (s - t_n)\nabla^2 U(X^h_s)\right)^{-1} \left( \nabla^3 U(X^h_s) :  \left(I + (s - t_n)\nabla^2 U(X^h_s)\right)^{-2}\right)  \, ds \\
         &\quad + \sqrt{2} \left(I + (s - t_n)\nabla^2 U(X^h_s)\right)^{-1} dW_s.
\end{aligned}
\]
Consequently, the Fokker-Planck equation \eqref{eq:overdampedFP} holds.

\end{proof}

Note that by Lemma \ref{lmm:matrixbound}, the matrix $$\left(I + (s - t_n)\nabla^2 U(x)\right)^{-1}$$ 
is always well-defined and uniformly bounded for all $x \in \mathbb{R}^d$. Then the derived SDE \eqref{eq:explicitSDE} is well-defined and facilitates the application of tools from stochastic analysis.

Moreover, based on the moment bound above, and combining with the polynomial bound assumed in Assumption \ref{ass1} above, it is easy to see that $\nabla^k U(X^h_t)$ has bounded $L^p$ norm ($k \leq 5$, $p \geq 1$), which will be repeatedly used in our analysis.

It is also well-known that $\rho_t$ associated with \eqref{eq:overdamped} satisfies a Fokker-Placnk equation given by
\begin{equation}\label{eq:overdampedFP}
    \partial_s \rho_s = \nabla \cdot (\nabla U \rho_s) + \Delta \rho_s.
\end{equation}
The two PDEs \eqref{eq:numericalFP} and \eqref{eq:overdampedFP} then enables us to estimate the relative entropy $\mathcal{H}(\rho^h_s \mid \rho_s)$ in the next section.

\section{Error Estimate in Relative Entropy}\label{sec:error}

In this section, we derive a relative entropy error bound for the iLMC discretization with second-order accuracy. The analysis is based on the continuous-time interpolation and the corresponding Fokker-Planck equation obtained in Section \ref{sec:interpolation}. Let us first recall the definition of the relative entropy $\mathcal{H}(\mu \| \nu)$ for two probability measures $\mu$, $\nu$ on $\mathbb{R}^d$:
\begin{equation*}
    \mathcal{H}(\mu | \nu):=\left\{\begin{array}{lr}
\int_E \log \frac{\mathrm{~d} \mu}{\mathrm{~d} \nu} \mathrm{~d} \mu, & \text { if } \mu \ll \nu, \\
\infty, & \text { otherwise }.
\end{array}\right.
\end{equation*}
In what follows, we prove our first main theorem, a relative entropy error bound for iLMC.

\begin{theorem}[Relative entropy error bound of iLMC]\label{thm:errormain}
Suppose Assumptions \ref{ass0}, \ref{ass1}, \ref{ass2} hold. Fix $T>0$. Let $\rho_t^h$, $\rho_t$ denote the laws of $X^h_t$, $X_t$ defined in \eqref{eq:overdamped}, \eqref{eq:interpolation}, respectively. Fix $\epsilon > 0$. There exists $\bar{h}>0$ and $C > 0$ independent of $h$ and $T$ ($C$ depends on $\epsilon$) such that for all $h \in (0,\bar{h})$, it holds that
\begin{equation}
    \sup_{0\leq s\leq T}\mathcal{H}\left(\rho^h_s \mid \rho_s \right) \leq CT^{3+\epsilon}h^2.
\end{equation}
\end{theorem}

\begin{proof}

Recall the Fokker-Planck equations for $\rho^h_s$ and $\rho_s$ defined in \eqref{eq:overdampedFP}, \eqref{eq:numericalFP}, respectively. Then for $s \in [t_n,t_{n+1})$, direct calculations yield   
% \tcb{do we need to write down more details of these calculations? probably in a lemma}
\begin{equation*}
\begin{aligned}
    \frac{d}{dt}\mathcal{H}(\rho^h_s | \rho_s) &=  \int_{\mathbb{R}^d}\rho^h_s 
	( b_h - (-\nabla U) ) \cdot \nabla\log \frac{ \rho^h_s}{\rho_s} d x 
	-  \int_{\mathbb{R}^d} (\Lambda_h - I) :(\nabla \rho^h_s\otimes \nabla\log \frac{ \rho^h_s}{\rho_s})  d x \\
&  \quad -  \int_{\mathbb{R}^d} \rho^h_s (\nabla \cdot \Lambda_h )
	\cdot \nabla \log \frac{ \rho^h_s}{\rho_s} d x 
- \int_{\mathbb{R}^d}\rho^h_s\left|\nabla\log \frac{ \rho^h_s}{\rho_s}\right|^{2} d x.
\end{aligned}
\end{equation*}
% \begin{equation*}
% \begin{aligned}
%     \frac{d}{dt}\mathcal{H}(\rho^h_s | \rho_s) &=  \int_{\mathbb{R}^d}\rho_s 
% 	( b_h - b ) \cdot \nabla\frac{ \rho^h_s}{\rho_s} d x 
% 	-  \int_{\mathbb{R}^d} (\Lambda_h - I) :(\nabla \rho^h_s\otimes \nabla\log \frac{ \rho^h_s}{\rho_s})  d x \\
% &  \quad -  \int_{\mathbb{R}^d} \rho_s (\nabla \cdot \Lambda_h )
% 	\cdot \nabla\frac{ \rho^h_s}{\rho_s} d x 
% - \int_{\mathbb{R}^d}\rho^h_s\left|\nabla\log \frac{ \rho^h_s}{\rho_s}\right|^{2} d x.
% \end{aligned}
% \end{equation*}
By Young's inequality,
\begin{equation*}
\begin{aligned}
    \frac{d}{dt}\mathcal{H}(\rho^h_s | \rho_s) &\leq \mathbb{E}|(b_h- (-\nabla U))(X^h_s)|^2 + \mathbb{E}|(\Lambda_h-I)\cdot\nabla \log \rho^h_s(X^h_s)|^2 + \mathbb{E}|(\nabla \cdot \Lambda_h)(X^h_s)|^2\\
    &=: I_1 + I_2 + I_3.
\end{aligned}
\end{equation*}
We aim to show that $I_i \lesssim h^2$ for $i=1,2,3$ up to time $T$. Consequently, the relative entropy error of iLMC is of second-order, i.e. 
\begin{equation*}
    \sup_{0\leq s \leq T}\mathcal{H}(\rho^h_s | \rho_s) \leq C(T) h^2.
\end{equation*}

\textbf{Estimate of $I_1$:}

By definition and Lemma \ref{lmm:matrixbound}, and note that the $p$-th moment for $\nabla^k U(X^h_s)$ ($k=1,2,3$) is uniformly bounded (recall the discussion after Proposition \ref{prop:explicitSDE}), for $h < \frac{1}{2M}$ one has
\begin{equation*}
\begin{aligned}
    &\quad \mathbb{E}\left|b_h(s,X_s^h) - \left(-\nabla U(X_s^h)\right)\right|^2\\
    &\leq 2\mathbb{E}\left|\nabla U(X_s^h) \left(I - \left(I + (s-t_n) \nabla^2 U(X_s^h) \right)^{-1}\right) \right|^2 + 16 h^2 \mathbb{E}|\nabla U^3(X_s^h)|^2\\
    &= 2\mathbb{E}\left|\nabla U(X_s^h) \left(I + (s-t_n) \nabla^2 U(X_s^h) \right)^{-1} (s-t_n) \nabla^2 U(X_s^h) \right|^2 + 16 h^2 \mathbb{E}|\nabla U^3(X_s^h)|^2\\
    &\leq Ch^2,
\end{aligned}
\end{equation*}
where the positive constant $C$ is independent of $s$, $T$ and $h$.

\textbf{Estimate of $I_2$:}

By definition of $\Lambda_h$ and Young's inequality, one has for $h < \frac{1}{2M}$ and any $\epsilon > 0$,
\begin{equation}\label{eq:I2Lp}
    \begin{aligned}
        &\quad\mathbb{E}\left|\left(\Lambda_h(s,X_s^h) - I\right) \cdot \nabla \log \rho^h_s(X_s^h)\right|^2 \\
        &=  (s-t_n)^2\mathbb{E}\left|\left(I + (s - t_n)\nabla^2 U(X^h_s)\right)^{-2} \nabla^2 U(X^h_s)  \left(2I +  (s-t_n)\nabla^2U(X^h_s)\right) \cdot \nabla \log \rho^h_s(X_s^h)\right|^2\\
        &\leq Ch^2\left( 1 + \mathbb{E}\left|\nabla \log \rho^h_s(X_s^h) \right|^{2 + \epsilon}\right).
    \end{aligned}
\end{equation}
Here we have used the polynomial bound for $\nabla^2 U$ and the moment bound for $X^h$, and the positive constant $C$ is independent of $s$, $T$ and $h$ but may depend on the positive constant $\epsilon$.

% \tcb{remark: we CANNOT obtain an a.e. bound for $|\Lambda_h - I|$ (because when $\nabla^2 U$ grows large, it is approximately $h^2\cdot h^{-1}$), so we have to to use the polynomial bound and taking its expectation. Consequently, we need an $L^p$ estimate instead of an estimate for the Fisher information.}

\textbf{Estimate of $I_3$}

Using the definition of $\Lambda_h$ again, for $h < \frac{1}{2M}$, one has
\begin{equation*}
    \mathbb{E}\left|\nabla \cdot \Lambda_h (s, X^h_s) \right|^2 \leq h^2\mathbb{E}\left[\left|\nabla^3 U(X^h_s) \right|^2\left|\left(I + (s-t_n) \nabla^2U(X^h_s) \right)^{-1}\right|^6\right] \leq Ch^2.
\end{equation*}
We have used the polynomial bound for $\nabla^3 U$ and the moment bound for $X^h$ , and the positive constant $C$ is independent of $s$, $T$ and $h$.

Finally, combining the estimates for $I_1$ -- $I_3$, one has
\begin{equation}\label{eq:combine123}
 \frac{d}{dt}\mathcal{H}(\rho^h_s | \rho_s) \leq Ch^2 \left(1 + \mathbb{E}\left|\nabla \log \rho^h_s(X_s^h) \right|^{2+\epsilon} \right).
\end{equation}
We prove in Proposition \ref{prop:nablalog} below that there exists $C>0$, $\ell_0 \geq 1$ that independent of $h$, $T$ such that
\begin{equation*}
    |\nabla \log \rho^h_t(x)| \leq CT\left(1 + |x|^{\ell_0} \right),\quad\forall x\in\mathbb{R}^d,\quad \forall t \in [0,T].
\end{equation*}
Consequently, using the moment bound for $X^h$, one knows that there exists $C>0$ that depends on $T$ such that
\begin{equation*}
    \mathbb{E}\left|\nabla \log \rho^h_s(X_s^h) \right|^{2 + \epsilon} \leq CT^{2 +\epsilon},\quad \forall s \in [0,T].
\end{equation*}
Combining this with \eqref{eq:combine123} gives the desired result.

\end{proof}

\begin{remark}
Note that the relative error bound above is valid only in a finite time horizon, and the main reason is that: the estimate for $\nabla \log \rho^h_s$ in Section \ref{sec:nablalogrho} is not uniform-in-time. It might be possible to improve this result to the long-time regime using some othe advanced tools, and we leave it as future work. Also, assuming some additional conditions such as the log-Sobolev inequality, the relative entropy estimate then implies a (finite-time) first-order convergence under Wasserstein distances due to classic transport inequalities \cite{otto2000generalization,talagrand1991new,bolley2005weighted,pinsker1963information}. On the other hand, provided with a Wasserstein contraction result which we will establish in Section \ref{sec:ergodicity} below, we are able to extend the convergence to a uniform-in-time one. We provide more details in Section \ref{sec:longtime} below.
\end{remark}

\begin{remark}
The $L^p$ ($p > 2$) bound required in \eqref{eq:I2Lp} during the proof cannot be reduced to a $L^2$ one (so that one only needs to study the Fisher information instead, which is much easier to control (see for instance \cite[Section 3]{li2022sharp}, \cite[Section 5]{mou2022improved})). The reason is that $\Lambda_h - I$ cannot be bounded pointwisely. We can see this from the second line of \eqref{eq:I2Lp}: when $|(s-t_n)\nabla^2 U(X^h_s)|$ goes to infinity, $|\Lambda_h - I|$ is approximately of order $h$ rather than desired $h^2$. Therefore, we can only apply Young's (or H\"older's) inequality, and then an $L^{2 + \epsilon}$ ($\epsilon > 0$) bound is required for $\nabla \log \rho^h_s(X_s^h)$.
\end{remark}

\section{Gradient estimate}\label{sec:nablalogrho}

In this section, we prove an $L^p$ ($p > 2$) bound for the random variable $\nabla \log \rho^h_s (X^h_s)$ ($s \in [0,T]$), which is used in the proof of Theorem \ref{thm:errormain}.
The construction and derivation below is not so novel in literature (see for instance \cite{bernstein1906generalisation,bernstein1910generalisation,li1986parabolic, du2024collision, feng2023quantitative, ju2025modified}), and much of our calculation follows \cite[Section 4]{du2024collision}.

Before the detailed estimation, let us first give a high-level overview of our technique. Basically, we prove a polynomial upper bound for $\nabla \log \rho^h$ (recall that $\rho^h$ solves the Fokker-Planck equation \ref{eq:numericalFP}) via Bernstein method for gradient estimate. Notably, below we successfully obtain the gradient estimate under the non-uniform-elliptic settings (recall that the diffusion coefficient in the SDE \eqref{eq:explicitSDE} is not uniformly bounded from below). See more discussions in Remark \ref{rmk:elliptic} below.

\paragraph{Technique overview} Our analysis is conducted following the three main steps:

\textbf{STEP 1.} Our proof begins with a Cole-Hopf transformation, which results in a Hamilton-Jacobi equation. In detail, letting 
\begin{equation}\label{eq:defu}
   u(t,x) := \log \rho^h_t(x) / M_0 \leq 0,
\end{equation}
where $M_0 := \exp(C_3' T)$ such that $\rho^h_t(x) \leq M_0$ for all $t \in [0,T]$ and $x \in \mathbb{R}^d$, see details in Lemma \ref{lmm:logrho} below.  It is then easy to see that $u$ satisfies a Hamilton-Jacobi equation
\begin{equation}\label{eq:HJ}
    \partial_t u=a:\left(\nabla^2 u+\nabla u \otimes \nabla u\right)+b \cdot \nabla u+c,
\end{equation}
where 
\begin{equation}\label{eq:abcovervoew}
    a:= \Lambda_h,\quad b:= -b_h+\nabla \cdot \Lambda_h,\quad c:= -\nabla \cdot b_h+\nabla^2: \Lambda_h.
\end{equation}
Note that under Assumption \ref{ass2}, for all $t>0$, the $\rho^h_t$ is positive everywhere due to the positivity of the diffusion matrix, so that one can take log. Also, $\rho^h_t$ is piecewise smooth in time, and $\partial_t\rho^h$ (or $\partial_t u$) could be discontinuous in time direction. This however will not affect the proof of the gradient estimates.

\textbf{STEP 2.} Then, we perform a gradient estimate using a Bernstein-type estimate, and obtain an inequality of the form 
\begin{equation}\label{eq:nablaucontroled}
    |\nabla u(x)| \leq \mathcal{P}(x)(1+|u(x)|),
\end{equation}
where $\mathcal{P}(x)$ is some polynomial of $x$. To obtain the inequality \eqref{eq:nablaucontroled}, we construct an auxiliary function
\begin{equation}\label{eq:gconstruction}
    g := \frac{|\nabla u|^2}{(1-u)^2}.
\end{equation}
% \tcb{Note that by Lemma \ref{lmm:logrho} below, $\rho^h_t$ is upper bounded for $t \in [0,T]$. So one can assume without loss of generality that $\rho_t^h \le 1$ by scaling and thus $u\le 0$.}
Then, via a sequence of straightforward (but a bit tedious) calculation, one obtains an estimate of the form
$\mathcal{A}g \gtrsim g^2 - \mathcal{P}(x)$, where $\mathcal{A}$ is a nonnegative operator of the form $\mathcal{A} g = -\partial_t g+ a^{ij} \partial_{ij} g + \bar{b}^i \partial_i g$ (Einstein summation convention is used here and in the rest of the paper), and the matrix $(a^{ij})$ is positive definite everywhere (but not necessarily uniformly). Then after a maximum principle type argument, one obtains \eqref{eq:nablaucontroled}. In particular, in the case where $g$ attains its maximum in the interior of the parabolic domain, at the maximum point, one has $0 \geq \mathcal{A}g \gtrsim g^2 - \mathcal{P}(x)$, so $g$ is still bounded by a polynomial. 

Throughout the analysis in this section, the construction of the auxiliary function $g$ in \eqref{eq:gconstruction} as well as the lower bound for $\mathcal{A}g$ above is crucial. We summarize the result in the following proposition, whose detailed proof is postponed to the end of this section:
\begin{proposition}\label{prop:Ag}
% Let $u:=\log \rho^h$ \tcr{for this proposition to hold, need $\rho^h\le 1$ so you may define directly $u=\log(\rho^h/M)$}. 
Suppose Assumption \ref{ass0} holds. Recall the functions $u$, $g$, $a$, $b$, $c$ defined in \eqref{eq:HJ} -- \eqref{eq:gconstruction} above. Define the nonnegative operator $\mathcal{A}$ by
\begin{equation}
    \mathcal{A}(g) := a^{ij} \partial_{ij} g - \partial_t g + b^i \partial_i g - 3a^{ij} \frac{\partial_j u}{1-u}\partial_i g + 2a^{ij} \partial_j u \partial_i g.
\end{equation}
Then, 
\begin{equation}
    \mathcal{A}(g) \geq \frac{\lambda}{2} \frac{\left|\nabla^2 u\right|^2}{(1-u)^2}+\frac{\lambda}{2}(1-u) g^2  - M_1 (1-u)  (g+1),\quad \forall x\in \mathbb{R}^d,
\end{equation}
where $\lambda(x) > 0$ is the smallest eigenvalue of $a(x)$, and
\begin{equation}
    M_1 := 2|c| + 2\lambda^{-1}|\nabla a|^2 + 2|\nabla b| +  2|\nabla c|.
\end{equation}
\end{proposition}

\begin{remark}\label{rmk:elliptic}
As we can see from the proof of Proposition \ref{prop:nablalog} below, a completely rigorous analysis involves a cut-off function that vanishes outside some neighborhood of a fixed point $x^{*}$. Moreover, to our knowledge, most existing similar results requires the operator $\mathcal{A}$ above is uniformly elliptic (i.e. the smallest eigenvalue of $(a^{ij})$ has a uniform lower bound in $\mathbb{R}^d$). However, in our setting, $(a^{ij})$ is only globally positive definite and the eigenvalues do not have a uniform lower bound. This in fact would not bring too much influence to our derivation -- we only make use of two facts regarding $(a^{ij})$:
\begin{enumerate}
    \item the inverse of the smallest eigenvalue of $(a^{ij})$ (which is in fact $(1 + (s-t_n)|\nabla^2 U(x)|)^2$) has a polynomial upper bound;
    \item $(a^{ij})$ is positive definite everywhere (consequently, when $g$ attains maximum in the interior of the parabolic domain, one has $a^{ij}\partial_{ij} g \leq 0$).
\end{enumerate}
\end{remark}

\textbf{STEP 3.} Finally, by proving a tail estimate of the numerical density of the form \eqref{eq:tailrhoh},
% \begin{equation*}
%     \exp(-C_1'T)\exp(-C_2 |x|^{\ell_1}) \leq \rho^h_t(x) \leq \exp(C_3'T)\exp(-\gamma U(x)),
% \end{equation*}
one knows that $|u|$ has a polynomial upper bound. Combining this with the estimate \eqref{eq:nablaucontroled}, one obtains that $|\nabla u|$ has a polynomial upper bound.

In what follows, we give the details of our derivation. We first need the following lemma describing the tail of $\rho^h_t$:
\begin{lemma}\label{lmm:logrho}
Suppose Assumptions \ref{ass0}, \ref{ass1}, \ref{ass2} hold. Fix $T>0$. Recall the constants $C_1, C_2, C_3, \gamma, \ell_1$ in Assumption \ref{ass2}. Then there exist $C_1',C_3' > 0$ such that
\begin{equation}\label{eq:tailrhoh}
    \exp(-C_1'T)\exp(-C_2 |x|^{\ell_1}) \leq \rho^h_t(x) \leq \exp(C_3'T)\exp(-\gamma U(x)),\quad \forall t \in [0,T].
\end{equation}
\end{lemma}

The proof of Lemma \ref{lmm:logrho} is relatively straightforward due to the maximal principle. We provide a detailed proof of \ref{lmm:logrho} in Appendix \ref{app:nablalogrho}.

Based on Lemma \ref{lmm:logrho} and Proposition \ref{prop:Ag} above, we are then able to derive a polynomial upper bound of $\nabla \log \rho^h_t(x)$. As mentioned in the technique overview, the key estimate is Proposition \ref{prop:Ag}. We will move the tedious proof for Proposition \ref{prop:nablalog} to Appendix \ref{app:nablalogrho} while give a detailed derivation for Proposition \ref{prop:Ag} at the end of this section.
% The basic steps are as follows: via a Bernstein-type analysis, We first establish an inequality of the form $|\nabla \log \rho^h| \leq \mathcal{P}(x)|\log \rho^h|$ where $\mathcal{P}(x)$ is some polynomial of  $x$; then, Lemma \ref{lmm:logrho} tells that $|\log \rho^h|$ is also bounded by a polynomial of $x$, which then leads to the conclusion.

\begin{proposition}\label{prop:nablalog}
Suppose Assumptions \ref{ass0}, \ref{ass1}, \ref{ass2} hold.  Then for any fixed $T>0$, there exist $C>0$  , $\ell'_0, \ell_0'' \geq 1$ independent of $t$, $x$, $h$ and $T$ such that
\begin{equation}\label{eq:nablalogandlog}
    |\nabla \log \rho^h_t(x)| \leq C\left(1 + |x|^{\ell'_0} \right)(1+ T + |\log \rho^h_t(x)|),\quad\forall x\in\mathbb{R}^d,\quad \forall t \geq 0,
\end{equation}
and consequently, 
\begin{equation}\label{eq:nablalogrhopoly}
    |\nabla \log \rho^h_t(x)| \leq CT\left(1 + |x|^{\ell_0''} \right),\quad \forall t \in [0,T].
\end{equation}
\end{proposition}

\begin{proof}[Proof of Proposition \ref{prop:Ag}]
After straightforward calculations, one has  
\begin{equation*}
    \begin{aligned}
\mathcal{A}(g) = & 2 a^{i j} \frac{\partial_{ik}u \partial_{jk}u}{(1-u)^2}+2 a^{i j} \frac{\partial_{ik}u \partial_{k}u \partial_{j}u}{(1-u)^3}+2 a^{i j} \frac{\partial_{k}u \partial_{k}u \partial_{i}u \partial_{j}u}{(1-u)^3} \\
& -2 c \frac{\partial_{k}u \partial_{k}u}{(1-u)^3}-2 \frac{\partial_{k}u\left(\partial_{k}a^{i j}\left(\partial_{ij}u+\partial_{i}u \partial_{j}u\right)+\partial_{k}b^i \partial_{i}u+\partial_{k}c\right)}{(1-u)^2}.
\end{aligned}
\end{equation*}
Note that by Lemma \ref{lmm:matrixbound}, $a$ is positive definite. So one has
\begin{equation*}
    a^{i j} \frac{\partial_{ik}u \partial_{jk}u}{(1-u)^2}+2 a^{i j} \frac{\partial_{ik}u \partial_{k}u \partial_{j}u}{(1-u)^3}+a^{i j} \frac{\partial_{k}u\partial_{k}u \partial_{k}u \partial_{i}u \partial_{j}u}{(1-u)^4} \geq 0.
\end{equation*}
Then,
\begin{equation*}
\begin{aligned}
\mathcal{A} (g) & \geq a^{i j} \frac{\partial_{ik}u \partial_{jk}u}{(1-u)^2}+a^{i j} \frac{\partial_{k}u \partial_{k}u \partial_{i}u \partial_{j}u}{(1-u)^3}\\
&\quad-\frac{2 c}{(1-u)} g-2 \frac{\partial_{k}u\left(\partial_{k}a^{i j}\left(\partial_{ij}u+\partial_{i}u \partial_{j}u\right)+\partial_{k}b^i \partial_{i}u+\partial_{k}c\right)}{(1-u)^2} \\
& \geq \lambda \frac{\left|\nabla^2 u\right|^2}{(1-u)^2}+\lambda(1-u) g^2-2|c| g\\
&\quad - 2|\nabla a| \frac{g^{1/2}|\nabla^2 u|}{1-u} - 2|\nabla a| (1-u)g^{3/2}- 2|\nabla b|g - 2|\nabla c| \frac{g^{1/2}}{1-u}\\
&\geq \frac{\lambda}{2} \frac{\left|\nabla^2 u\right|^2}{(1-u)^2}+\frac{\lambda}{2}(1-u) g^2-2|c| g\\
&\quad - 2\lambda^{-1}|\nabla a|^2 g - 2\lambda^{-1}|\nabla a|^2 (1-u)g- 2|\nabla b|g - 2|\nabla c| (g+1)\\
& \geq \frac{\lambda}{2} \frac{\left|\nabla^2 u\right|^2}{(1-u)^2}+\frac{\lambda}{2}(1-u) g^2  - M_1 (1-u)  (g+1),
\end{aligned}
\end{equation*}
where the function $M$ is defined by
\begin{equation*}
    M_1 := 2|c| + 2\lambda^{-1}|\nabla a|^2 + 2|\nabla b| +  2|\nabla c|.
\end{equation*}
    
\end{proof}

\section{Geometric ergodicity via a reflection-type coupling method}\label{sec:ergodicity}

In this section, We prove the second main result: geometric ergodicity of iLMC, using a reflection-type continuous-discrete coupling method with a Lyapunov function defined by
\begin{equation}\label{eq:Lyapunovdef}
    f(r) = \int_0^r e^{-C_f (r' \wedge R_f)} dr', \quad r \geq 0,
\end{equation}
and the associated Kantorovich-Rubinstein distance $W_f$ defined by
\begin{equation}\label{eq:Wf}
    W_{f}(\mu, \nu):=\inf _{\gamma \in \Pi(\mu, \nu)} \int_{\mathbb{R}^{d} \times \mathbb{R}^{d}}f(|x-y|) d \gamma.
\end{equation}
Recall that the Wasserstein-1 distance is defined by
\begin{equation}\label{eq:W1}
    W_{1}(\mu, \nu):=\inf _{\gamma \in \Pi(\mu, \nu)} \int_{\mathbb{R}^{d} \times \mathbb{R}^{d}}|x-y| d \gamma.
\end{equation}
Clearly, since $  e^{-c_f R_f} r\leq f(r) \leq r$ for all $r \geq 0$, contraction under $W_f$ is equivalent with contraction under $W_1$ to some extent.

\begin{theorem}[Wasserstein contraction of iLMC]\label{thm:contraction}
Suppose Assumption \ref{ass0} holds. Denote $\mu^h_n$, $\nu^h_n$ be the law of iLMC with step size $h$ at $n$-th iteration, with initial distributions $\mu_0$, $\nu_0$, respectively. Denote $R' = (4 + 16M/m)R$. Then for fixed small $h>0$, one can choose $R_f = 3R'$ and $c_f > C(R', M)$ such that
\begin{equation}
    W_f (\mu^h_n, \nu^h_n) \leq e^{-Cnh} W_f(\mu_0, \nu_0),
\end{equation}
where the positive constant $C$ is independent of $h$ and $n$.
Consequently, 
\begin{equation}
    W_1(\mu^h_n, \nu^h_n) \leq C_0e^{-Cnh} W_1(\mu_0, \nu_0),\quad C_0 := e^{c_f R_f}.
\end{equation}
\end{theorem}

A direct corollary of the Wasserstein contraction result in Theorem \ref{thm:contraction} is the following geometric ergodicity of iLMC:

\begin{corollary}[Geometric ergodicity of iLMC]
Suppose Assumption \ref{ass0} holds. Denote $\rho^h_n$ be the law of iLMC with step size $h$ at $n$-th iteration, with initial distribution $\rho_0$. Then for small $h$, the iLMC as a discrete-time Markov chain has a unique invariant measure $\pi^h$, and
\begin{equation}\label{W1ergodicity}
    W_1 (\rho^h_n, \pi^h) \leq C_0 e^{-Cnh} W_1(\rho_0, \pi^h).
\end{equation}
\end{corollary}

\begin{proof}
Similar proofs can also be found in related literature such as \cite{li2024geometric, li2025ergodicity}. By Theorem \ref{thm:contraction}, there exists $n_0 \in \mathbb{N}_{+}$  such that 
\begin{equation}
    W_1(\mu^h_{n_0}, \nu^h_{n_0}) \leq \frac{1}{2} W_1(\mu_0,\nu_0).
\end{equation}
Denote the corresponding transition kernel for $n$th iteration by $P_{n}$. Then, $\mu \mapsto \mu P_{n_0}$ is contractive. By Banach's contraction mapping theorem, there exists a fixed point $\pi_*$ satisfying
\begin{equation}
    \pi_* = \pi_* P_{n_0}.
\end{equation}
Then, by Markov property, $\pi^h := \frac{1}{n_0}\sum_{n=0}^{n_0-1} \pi_* P_n$ is the invariant measure of the iLMC iteration. 
Moreover, $\pi^h = \pi^h P_{n_0}$ for any invariant measure so that the invariant measure is unique by the contraction property of $P_{n_0}$. Besides, $\pi^h=\pi_*$. 

Letting $\nu^h_{nh} = \pi^h$ in Theorem \ref{thm:contraction}, \eqref{W1ergodicity} then follows.
\end{proof}

Next, we prove Theorem \ref{thm:contraction} via a reflection-type coupling method.
For any $h>0$, recall that we define the map $\Phi_h : \mathbb{R}^d \rightarrow \mathbb{R}^d$ by
\begin{equation}\label{eq:defphi}
    \Phi_h(x) = x + h \nabla U(x),\quad \forall x \in \mathbb{R}^d,
\end{equation}
and under Assumptions \ref{ass0}, for small $h$, $\Phi_h$ is proved to be a homeomorphism in Proposition \ref{prop:PhiLip} above. Also recall that we can rewrite iLMC \eqref{eq:ilmciteration} as
\begin{equation}\label{eq:iLMCPhi}
    X^h_{t_{n+1}} = \Phi^{-1}_h\left(X^h_{t_n} + \sqrt{2}\,\Delta W_n\right).
\end{equation}
In each iteration, the iLMC \eqref{eq:iLMCPhi} is in fact performed in two steps:
\begin{equation}
    X^h_{t_n} \xrightarrow{\text{diffusion step}} \tilde{X}^h_{t_{n+1}} \left(= X^h_{t_n} + \sqrt{2}\,\Delta W_n\right) \xrightarrow{\text{drift step}} X^h_{t_{n+1}} \left(= \Phi_h^{-1}(\tilde{X}^h_{t_{n+1}}) \right).
\end{equation}
The drift step is deterministic, and its evolution can be estimated using properties of the map $\Phi_h$ and its inverse (see Section \ref{sec:drift} below). For the diffusion part, our analysis is mainly based on a continuous-time reflection coupling (see Section \ref{sec:diffusion} below).

\subsection{Evolution of the drift step}\label{sec:drift}
Firstly, recall that under Assumption \ref{ass0}, we prove in Proposition \ref{prop:PhiLip} above that the following Lipschitz property of $\Phi_h^{-1}$ holds:
\begin{equation}
        \left| \Phi^{-1}_h(x) - \Phi^{-1}_h(y)\right| \leq
        \left\{
        \begin{aligned}
            & e^{-\frac{m}{4}h} |x-y|,\quad |x-y| > R',\\
            & e^{2Mh} |x-y|, \quad |x-y| \leq R'.
        \end{aligned}
        \right.
    \end{equation}

Furthermore, considering the Lyapunov function $f(\cdot)$, we are able to prove the following:

\begin{lemma}\label{lmm:fPhi}
    Suppose Assumptions \ref{ass0}  hold. Recall the function $f(\cdot)$ defined in \eqref{eq:Lyapunovdef}. Then for then for $R'$ in Proposition \ref{prop:PhiLip}, when $h \in (0,1/(2M)),$ it holds
    \begin{equation}\label{eq:fPhiclaim0}
        f\left(\left| \Phi^{-1}_h(x) - \Phi^{-1}_h(y)\right|\right) \leq
        \left\{
        \begin{aligned}
            & f(|x-y|) - \frac{m}{4} h|x-y|f'(|x-y|),\quad |x-y| > R',\\
            & f(|x-y|) + 2M h|x-y|f'(|x-y|), \quad |x-y| \leq R'.
        \end{aligned}
        \right.
    \end{equation}
    Consequently, for $C_1' = e^{-c_fR_f}\frac{m}{4}$ and $C_2' = e^{c_fR_f}2M$,
    \begin{equation}\label{eq:fPhiclaim1}
        f\left(\left| \Phi^{-1}_h(x) - \Phi^{-1}_h(y)\right|\right) \leq
        \left\{
        \begin{aligned}
            & e^{-C_1'h} f(|x-y|),\quad |x-y| > R',\\
            & e^{C_2'h} f(|x-y|), \quad |x-y| \leq R'.
        \end{aligned}
        \right.
    \end{equation}
\end{lemma}

\begin{proof}
Since $f$ is concave,
\begin{equation*}
   f\left(\left| \Phi^{-1}_h(x) - \Phi^{-1}_h(y)\right|\right) \leq f(|x-y|) + f'(|x-y|) \left(\left| \Phi^{-1}_h(x) - \Phi^{-1}_h(y)\right| - |x-y| \right).
\end{equation*}
\eqref{eq:fPhiclaim0} then follows due to Proposition \ref{prop:PhiLip} and the fact that $f' > 0$. \eqref{eq:fPhiclaim1} is a direct corollary of \eqref{eq:fPhiclaim0}, since $e^{-c_f R_f} \leq f'(r) \leq 1$, and $e^{-c_f R_f}r \leq f(r) \leq r, \forall r \geq 0$.

\end{proof}

Lemma \ref{lmm:fPhi} reveals the contraction effect brought by the far-field confining drift term $\nabla U$, and plays the key role in the proof of Theorem \ref{thm:contraction}. In particular, the claim \eqref{eq:fPhiclaim0} will be used to derive \eqref{eq:case1eq2} (in Case 1 below) and \eqref{eq:uselmm52} (in Case 2 below); the claim \eqref{eq:fPhiclaim1} will be used in \eqref{eq:usefphiclaim1} and arguments before \eqref{eq:endcase1} below.

\subsection{Evolution of the diffusion step: a reflection-type coupling}\label{sec:diffusion}

The analysis for the diffusion step is based on a reflection-type coupling approach. Firstly, let us introduce the construcsted coupling $\left(\left(X^h_{t_{n}}\right)_{n=0}^{\infty}, \left(Y^h_{t_{n}}\right)_{n=0}^{\infty}  \right)$ with the initial distributions $X^h_0 \sim \mu_0$, $Y^h_0 \sim \nu_0$, and $(X^h_0, Y^h_0)$ is the optimal coupling so that $W_1 (\mu_0, \nu_0) = \mathbb{E}|X^h_0 - Y^h_0|$ (note that such $(X^h_0, Y^h_0)$ can always be found due to standard optimal transport theory \cite{villani2008optimal}).
At $n$-th iteration, $\left(\left(X^h_{t_{n}}\right)_{n=0}^{\infty},\left(Y^h_{t_{n}}\right)_{n=0}^{\infty}  \right)$ is evolving according to the followings:
\begin{equation}
    \begin{aligned}
        &\tilde{X}^h_t = X^h_{t_n} + \sqrt{2}\int_{t_n}^{t} dW_s, \,\, t \in [t_n, t_{n+1}],\quad X_{t_{n+1}}^h = \Phi_h^{-1}\left(\tilde{X}^h_{t_{n+1}}\right),\\
        &\tilde{Y}^h_t =\left\{
        \begin{aligned}
            &Y^h_{t_n} + \sqrt{2}\int_{t_n}^{t} \left(I_d - 2 e_s^{\otimes 2} \right)\cdot dW_s,\,\,t<\tau,\\
            &\tilde{X}^h_t,\,\, t\geq \tau,
        \end{aligned}
        \right.
        \,\,t \in [t_n, t_{n+1}],\quad Y_{t_{n+1}}^h = \Phi_h^{-1}\left(\tilde{Y}^h_{t_{n+1}}\right),
    \end{aligned}
\end{equation}
where $W_s$ above denotes the same Brownian motion,
\begin{equation}
    e_t := \frac{\tilde{X}^h_t - \tilde{Y}^h_t}{|\tilde{X}^h_t - \tilde{Y}^h_t|},
\end{equation}
and the stopping time $\tau$ is define by
\begin{equation}
    \tau := \inf\{t \geq 0: \tilde{X}^h_t = \tilde{Y}^h_t \}.
\end{equation}
Clearly, $\left(\left(X^h_{t_{n}}\right)_{n=0}^{\infty}, \left(Y^h_{t_{n}}\right)_{n=0}^{\infty}  \right)$ are two couplied copies of iLMC. Also note that the stopping can only happen during the diffusion step since $\Phi_h^{-1}$ is a homomorphism for $h < 1/(2M)$. Also, if $\tau \in [t_n,t_{n+1})$, it is easy to see that $X^h_{t_m} = Y^h_{t_m}$ for all $m > n$.

In what follows, let us fix $n \in \mathbb{N}$ and focus on the one-step evolution. Denote $Z^h_{t_n}:= X^h_{t_n} - Y^h_{t_n}$ and $\tilde{Z}^h_t := \tilde{X}^h_t - \tilde{Y}^h_t$ for $t \in [t_n, t_{n+1}]$. Clearly,
\begin{equation*}
    \tilde{Z}^h_{t} = X^h_{t_n} - Y^h_{t_n} + 2\sqrt{2}\int_{t_n \wedge \tau}^{t\wedge \tau} \frac{(\tilde{Z}_s^h)^{\otimes 2}}{|\tilde{Z}_s^h|^2 }\cdot dW_s.
\end{equation*}
Then, since
\begin{equation*}
    \begin{aligned}
& \nabla f(|x|)=f^{\prime}(|x|) \frac{x}{|x|}, \quad
 \nabla^2 f(|x|)=f^{\prime \prime}(|x|) \frac{x \otimes x}{|x|^2}+f^{\prime}(|x|) \frac{1}{|x|}\left(I-\frac{x \otimes x}{|x|^2}\right),
\end{aligned}
\end{equation*}
Dykin's formula directly gives
\begin{lemma}\label{lmm:ddtEfZ}
For all $t \geq 0$,
\begin{equation}
\frac{d}{d t} \mathbb{E} f(|\tilde{Z}^h_t|)=4\mathbb{E} \left[ f^{\prime \prime}(|\tilde{Z}^h_t|) \textbf{1}_{\{t \leq \tau\}}\right]=-4  c_f \mathbb{E} \left[e^{-c_f(|\tilde{Z}^h_t| \wedge R_f)} \textbf{1}_{\{|\tilde{Z}^h_t| \leq R_f\}} \textbf{1}_{\{t \leq \tau\}}\right].
\end{equation}
\end{lemma}

Combining Lemma \ref{lmm:matrixbound} and Lemma \ref{lmm:ddtEfZ}, we know that when $|\tilde{Z}^h_{t_{n+1}}|$ is large, the drift step is contractive; when $|X^h_{t_{n}} - Y^h_{t_{n}}|$ is small, the diffusion step is contractive. Moreover, it can be shown that for any $t, s \in [t_m,t_{n+1}]$, the difference of $\tilde{Z}^h_{s}$ and $\tilde{Z}^h_{t}$ is subGaussian:

\begin{lemma}\label{lmm:subgaussian}
Denote
\begin{equation}
    \zeta_t := \int_{t_n \wedge \tau}^{t\wedge \tau} \frac{(\tilde{Z}_s^h)^{\otimes 2}}{|\tilde{Z}_s^h|^2 }\cdot dW_s,\quad t\in[t_n,t_{n+1}].
\end{equation}
Then for any $t, s \in [t_n,t_{n+1}]$, and $a \geq 0$, there exists a positive constant $C$ independent of $t$, $s$, $h$, $a$ such that
\begin{equation}
    \mathbb{P}\left(|\zeta_t - \zeta_s| \geq a \right) \leq 2 \exp \left( -C h^{-1} a^2\right).
\end{equation}
\end{lemma}

The main reason that Lemma \ref{lmm:subgaussian} holds is that the covariance matrix has unit norm. We refer the readers to [Li, Liu, Wang, 2024] for a similar proof. We also provide a detailed derivation in Appendix \ref{app:ergodicitylemma}.

Now mainly based on Lemma \ref{lmm:fPhi}, Lemma \ref{lmm:ddtEfZ} and Lemma \ref{lmm:subgaussian}, we are able to combine the drift step and diffusion step and prove the $W_f$-contraction. The detailed derivation also involves some technical lemmas estimating some small-probability events. We move these tedious derivations to Appendix \ref{app:ergodicitylemma}.

\begin{proof}[Proof of Theorem \ref{thm:contraction}]

Fix $n \in \mathbb{N}$. It suffices to prove the following one-step contraction for the coupling $\left(\left(X^h_{t_{n}}\right)_{n=0}^{\infty}, \left(Y^h_{t_{n}}\right)_{n=0}^{\infty}  \right)$:

\begin{equation*}
    \mathbb{E}f(|X^h_{t_{n+1}} - Y^h_{t_{n+1}}|) \leq (1 - Ch)\mathbb{E}f(| X^h_{t_n} - Y^h_{t_n}|)
\end{equation*}
for some positive constant $C$ independent of $h$ and $n$

% (following \cite{li2025ergodicity} ... )

Recall that we choose $R_f = 3R' = 3(4 + 16M / m)R$ and $c_f>0$ is a large constant to be determined below. Denote $Z^h_{t_n}:= X^h_{t_n} - Y^h_{t_n}$ for $n = 0,1,2,\dots$. Fix $n \in \mathbb{N}$. Fix a small $\delta \in (0, 1/2)$. We decompose the whole probability space into the following parts:
\begin{equation*}
    \Omega_1 := \left\{|Z^h_{t_n}| < h^{\frac{1}{2}-\delta} \right\},\quad
    \Omega_2 := \left\{ h^{\frac{1}{2}-\delta} \leq |Z^h_{t_n}| \leq 2R'\right\},\quad
    \Omega_3 := \left\{|Z^h_{t_n}| > 2R' \right\}.
\end{equation*}
Note that the main reason for the choice of $h^{1/2-\delta}$ is the subGaussian tail in Lemma \ref{lmm:subgaussian}.

\textbf{Case 1. Consider $\Omega_1$.} The main challenge in this case is that the probability of $t < \tau$ may not be close to $1$. Denote the event
\begin{equation*}
    F_1(t)=\left\{\exists s \in\left[t_n, t\right]:\left|\tilde{Z}^h_s\right|=2 h^{1 / 2-\delta}\right\},\quad t \in [t_n,t_{n+1}].
\end{equation*}
Clearly, by continuity of Brownian motion, on $F_t(t)^c$, one always has $|\tilde{Z}^h_t| \leq R_f$. Then since $\exp(-2c_fh^{\frac{1}{2}-\delta}) \geq \frac{3}{4}$, similarly as in Lemma \ref{lmm:ddtEfZ}, after It\^o's calculus one can obtain that
\begin{equation*}
    \frac{d}{d t} \mathbb{E}\left[ \textbf{1}_{\Omega_1} f\left(\left|\tilde{Z}^h_t\right|\right)\right] \leq-3  c_f \mathbb{E} \left[\textbf{1}_{\Omega_1 \cap F_1(t)^c} \textbf{1}_{\{t<\tau\}}\right].
\end{equation*}
We show in Lemma \ref{lmm:technical1} below that
\begin{equation*}
    \mathbb{E}\left[ \textbf{1}_{\Omega_1 \cap F_1(t)^c} 1_{\{t<\tau\}} \right]\geq\left(1-\eta_1\left(h^{1 / 2-\delta}, 2 h^{1 / 2-\delta}, h\right)\right) \mathbb{E} \left[\tone_{\Omega_1} \tone_{\{t<\tau\}}\right].
\end{equation*}
Hence, for $h$ small enough, one has
\begin{equation}\label{eq:case1eq1}
    \frac{d}{d t} \mathbb{E}\left[ \tone_{\Omega_1} f(|\tilde{Z}^h_t|)\right] \leq-2 c_f \mathbb{E} \left[\tone_{\Omega_1} \tone_{\{t<\tau\}} \right]\leq - c_f \mathbb{E} \left[\tone_{\Omega_1} \tone_{\{t<\tau\}}\right]- c_f \mathbb{E} \left[\tone_{\Omega_1} \tone_{\{t_{n+1}<\tau\}}\right].
\end{equation}
Intuitively, we can then make use of the latter term to control the drift step. In fact, by \eqref{eq:fPhiclaim0} in Lemma \ref{lmm:fPhi}, and since $f' \in (0,1)$, one has
\begin{equation}\label{eq:case1eq2}
    \mathbb{E} \left[\tone_{\Omega_1} f\left(\left|Z^h_{t_{n+1}}\right|\right)\right]-\mathbb{E} \left[\tone_{\Omega_1} f\left(\left|\tilde{Z}^h_{t_{n+1}}\right|\right) \right]\leq 4M h  \mathbb{E} \left[\tone_{\Omega_1}\left|\tilde{Z}^h_{t_{n+1}}\right|\right].
\end{equation}
By the second claim in Lemma \ref{lmm:technical1}, one has
\begin{equation}\label{eq:caseqeq3}
\begin{aligned}
\mathbb{E} \left[\tone_{\Omega_1}\left|\tilde{Z}^h_{t_{n+1}}\right| \right]&\leq(1+\eta_2(h^{1 / 2-\delta}, 2 h^{1 / 2-\delta}, h)) \mathbb{E} \left[\tone_{\Omega_1}\left|\tilde{Z}^h_{t_{n+1}}\right| \tone_{\{|\tilde{Z}^h_{t_{n+1}}| \leq 2 h^{1 / 2-\delta}\}} \right]\\
&\leq 2 h^{1 / 2-\delta}  (1+\eta_2(h^{1 / 2-\delta}, 2 h^{1 / 2-\delta}, h)) \mathbb{E} \left[\tone_{\{\tau>t_{n+1}\}} \tone_{\Omega_1}\right].
\end{aligned}
\end{equation}
Here, $\eta_2(a,b,h):=\frac{6}{a}\left[b+\frac{h}{C (b-a)}\right] \exp \left(-\frac{C (b-a)^2}{2 h}\right)$, so for $\delta \in (0,1/2)$ and small $h$, $\eta_2$ is a small positive number dacaying to zero exponentially fast as $h$ vanishes. Hence, concluding \eqref{eq:case1eq1} -- \eqref{eq:caseqeq3}, one can choose $c_f$ large ($c_f > 16Mh^{1/2-\delta}$) such that
$$
\mathbb{E} \tone_{\Omega_1} f\left(\left|Z^h_{t_{n+1}}\right|\right) \leq \mathbb{E} \tone_{\Omega_1} f\left(\left|Z^h_{t_n}\right|\right)- c_f \int_{t_n}^{t_{n+1}} \mathbb{E} \tone_{\Omega_1} \tone_{\{t<\tau\}}.
$$
It is then remaining to handle the $\int_{t_n}^{t_{n+1}} \mathbb{E} \tone_{\Omega_1} \tone_{\{t<\tau\}}$ term.

Let $u(t):=\mathbb{E} \left[\tone_{\Omega_1} f(|\tilde{Z}^h_t|)\right]$. By \eqref{eq:case1eq1} and the similar argument in \eqref{eq:caseqeq3}, for $t<t_{n+1}$, one has
$$
\begin{aligned}
u(t) &\leq u\left(t_n\right)- c_f \int_{t_n}^t \mathbb{E} \left[\tone_{\Omega_1} \tone_{\{t<\tau \}}\right]dt
\leq u\left(t_n\right)-\frac{ c_f}{2 h^{1 / 2-\delta}} \int_{t_n}^{t_{n+1}} \mathbb{E} \left[\tone_{\Omega_1}\left|\tilde{Z}^h_t\right| \tone_{\{|\tilde{Z}^h_t| \leq 2 h^{1 / 2-\delta}\}}\right]dt \\
&\leq u\left(t_n\right)-\frac{ c_f}{2 h^{1 / 2-\delta}}\left(1-\eta_2\right) \int_{t_n}^t u(s) d s \leq u(t_n) - c\int_{t_n}^t u(s) ds,
\end{aligned}
$$
where have used the fact $\left|\tilde{Z}^h_t\right| \tone_{\{t<\tau\}}=\left|\tilde{Z}^h_t\right|$.
Then, if one directly applies the Gr\"onwall's inequality, the only remaining problem is: $\mathbb{E} \tone_{\Omega_1} f\left(\left|Z^h_{t_{n+1}}\right|\right) \neq u\left(t_{n+1}\right)$. In fact, we can resolve this by defining 
$$v(t):=\max \left\{\mathbb{E} \tone_{\Omega_1} f\left(\left|Z^h_{t_{n+1}}\right|\right), u(t)\right\}.$$ 
Since $\mathbb{E} 1_{\Omega_1} f\left(\left|Z_{t_{n+1}}\right|\right) \leq u\left(t_n\right)$ by the estimate above, one knows that $v$ is continuous and $v\left(t_n\right)=u\left(t_n\right)$. Obviously, $u(t)$ is monotonically decreasing and $\mathbb{E} \left[\tone_{\Omega_1} f\left(\left|Z^h_{t_{n+1}}\right|\right)\right] \leq 2 u\left(t_{n+1}\right)$ for small $h$ (recall \eqref{eq:fPhiclaim1} in Lemma \ref{lmm:fPhi}). Hence, one has $u(s) \geq C v(s)$ for some universal positive constant $C$. Consequently,
$$
v(t) \leq v\left(t_n\right)-c \int_{t_n}^t v(s) d s.
$$
Therefore, by choosing large $c_f$ ($c_f > 16Mh^{1/2-\delta}$, and $c_f>CM^{\delta + 1/2}$ suffices due to $h<1/(2M)$), there exists positive $C$ independent of $h$ and $n$ such that
\begin{equation}\label{eq:endcase1}
     \mathbb{E}\left[f(|Z^h_{t_{n+1}}|)\textbf{1}_{\Omega_1}\right]\leq e^{-C h} \mathbb{E}\left[f(|Z^h_{t_n}|)\textbf{1}_{\Omega_1}\right].
\end{equation}

\textbf{Case 2. Consider $\Omega_2$.}
Note that on $\Omega_2$, by the subGaussian property in Lemma \ref{lmm:subgaussian}, $\tau \leq t$ will almost not happen. Here a main challenge is that $e^{-c_f\left|\tilde{Z}^h_t\right|}$ is not close to $1$ . Since $c_f$ is large, we cannot naively bound this by $e^{-2c_f R'}$ from below. To address this, we define $\mu_m:=m h^{1 / 2-\delta}$ and decompose $\Omega_2$ into the following parts:

$$
\Omega_{2, m}:=\left\{\mu_m \leq\left|Z_{t_n}\right|<\mu_{m+1}\right\}, \quad m=1,2, \cdots,\left\lceil 2 R' / h^{1 / 2-\delta}\right\rceil-1 .
$$

Similar to Lemma \ref{lmm:ddtEfZ}, during the diffusion step, one has
$$
\frac{d}{d t} \mathbb{E} \left[\tone_{\Omega_{2, m}} f(|\tilde{Z}^h_t|)\right]=-4  c_f \mathbb{E} \left[\tone_{\Omega_{2, m}} e^{-c_f|
\tilde{Z}^h_t|} \tone_{\{|\tilde{Z}^h_t| \leq R_f\}} \tone_{\{t \leq \tau\}}\right].
$$
Using the fact $e^{-C  h} \mu_m \leq\left|Z_{t_n}\right| \leq \mu_{m+1} e^{Ch} \leq 2R' e^{C h}$, one has (for $a \in (0, h^{1/2-\delta}$)
$$
\begin{aligned}
\mathbb{E}\left[\tone_{\Omega_{2, m}} e^{-c_f|\tilde{Z}^h_t|} \tone_{\{|\tilde{Z}^h_t| \leq R_f\}} \tone_{\{t \leq \tau\}}\right]& \geq e^{-c_f \mu_{m+1} e^{C  h}} \mathbb{E}\left[\tone_{\Omega_{2, m}} e^{-c_f|\tilde{Z}^h_t-Z^h_{t_n}}| \tone_{\{|\tilde{Z}^h_t| \leq R_f\}} \tone_{\{t \leq \tau\}}\right] \\
& \geq \frac{1}{2} \frac{e^{-c_f \mu_m}}{f(R_f)} \mathbb{E}\left[\tone_{\Omega_{2, m}} \tone_{\{|\tilde{Z}^h_t-Z^h_{t_n}|<a\}} f(|\tilde{Z}^h_t|)\right].
\end{aligned}
$$
Clearly, for small $h$, applying Lemma \ref{lmm:technical2} with $a=3 h^{1 / 2-\delta} / 8$, one has
$$
\mathbb{E} \left[\tone_{\Omega_{2, m}} e^{-c_f|\tilde{Z}^h_t|} \tone_{\{|\tilde{Z}^h_t| \leq R_f\}} \tone_{\{t \leq \tau\}}\right] \geq \frac{1}{4 f(R_f)} e^{-c_f \mu_m} \mathbb{E}\left[\tone_{\Omega_{2, m}} f(|\tilde{Z}^h_t|)\right].
$$
Hence, by Gr\"onwall's inequality, the diffusion step gives
\begin{equation*}
    \mathbb{E} \left[\tone_{\Omega_{2, m}} f(|\tilde{Z}^h_{t_{n+1}}|)\right] \leq \exp \left(-\beta^{-1} \frac{c_f}{f\left(R_1\right)} e^{-c_f \mu_m} h\right) \mathbb{E} \left[ \tone_{\Omega_{2, m}} f\left(\left|Z^h_{t_n}\right|\right)\right].
\end{equation*}

For the drift step, 
% \tcr{1111, need more careful bound in Lemma \ref{lmm:fPhi} to preserve the $e^{-c_f |x-y|}$ term ,done}
by \eqref{eq:fPhiclaim0} in Lemma \ref{lmm:fPhi},
\begin{multline}\label{eq:uselmm52}
\mathbb{E} \left[\tone_{\Omega_{2, m}} f\left(\left|Z^h_{t_{n+1}}\right|\right)\right]-\mathbb{E} \left[\tone_{\Omega_{2, m}} f\left(\left|\tilde{Z}^h_{t_{n+1}}\right|\right) \right] \leq 2Mh\mathbb{E}\left[\tone_{\Omega_{2, m}} f'\left(\left|\tilde{Z}^h_{t_{n+1}}\right| \right) \left|\tilde{Z}^h_{t_{n+1}}\right|\right]\\
= 2M h \mathbb{E} \left[e^{-c_f\left|\tilde{Z}^h_{t_{n+1}}\right| \wedge R_f} \left|\tilde{Z}^h_{t_{n+1}}\right|\tone_{\{t_{n+1}<\tau\}} \tone_{\Omega_{2, m}}\right].
\end{multline}
% It is clear that $\left|Z_s-Z_{n+1}^{-}\right| \leq C \tau$ and $\left|Z_s\right| \leq e^{C L \tau}\left|Z_{n+1}^{-}\right|$. Hence,
% $$
% \mathbb{E} e^{-c_f\left|Z_s\right| \wedge R_1} 1_{t_{n+1}<\tau}\left|Z_s\right| 1_{\Omega_{2, m}} \leq e^{C L \tau} e^{C c_f \tau} \mathbb{E} e^{-c_f\left|Z_{n+1}^{-}\right| \wedge R_1}\left|Z_{n+1}^{-}\right| 1_{\Omega_{2, m}}
% $$
By the second claim in Lemma \ref{lmm:technical1}, one has
\begin{multline*}
    \mathbb{E} \left[e^{-c_f|\tilde{Z}^h_{t_{n+1}}| \wedge R_f}\left|\tilde{Z}^h_{t_{n+1}}\right| \tone_{\Omega_{2, m}} \tone_{\{t_{n+1}<\tau\}}\right]\\
\leq\left(1+e^{-c_f R_f} \eta_2\left(\mu_{m+1}, R_f, h\right)\right) \mathbb{E} \left[e^{-c_f|\tilde{Z}^h_{t_{n+1}}|}\left|\tilde{Z}^h_{t_{n+1}}\right| \tone_{\{|\tilde{Z}^h_{t_{n+1}}| \leq R_f\}} \tone_{\Omega_{2, m}} \right].
\end{multline*}
Once can choose $h$ small such that $\eta_2 < 1$ which further implies $1 + e^{-c_f R_f} \eta_2 \leq 2$. Since $r / f(r)$ is increasing for $r \in (0, \infty)$, one further has
$$
\mathbb{E} \left[e^{-c_f|\tilde{Z}^h_{t_{n+1}}|}\left|\tilde{Z}^h_{t_{n+1}}\right| \tone_{\{|\tilde{Z}^h_{t_{n+1}}| \leq R_f\}} \tone_{\Omega_{2, m}}\right]
\leq \frac{R_f}{f\left(R_f\right)} \mathbb{E}\left[ e^{-c_f|\tilde{Z}^h_{t_{n+1}}|} f\left(|\tilde{Z}^h_{t_{n+1}}|\right) \tone_{\{|\tilde{Z}^h_{t_{n+1}}| \leq R_f\}} \tone_{\Omega_{2, m}}\right].
$$
Applying Lemma \ref{lmm:technical2} again,
$$
\begin{aligned}
&\quad \mathbb{E}\left[ e^{-c_f|\tilde{Z}^h_{t_{n+1}}|} f\left(|\tilde{Z}^h_{t_{n+1}}|\right) \tone_{\{|\tilde{Z}^h_{t_{n+1}}| \leq R_f\}} \tone_{\Omega_{2, m}} \right]\\
& \leq \eta_3 \mathbb{E} f\left(|\tilde{Z}^h_{t_{n+1}}|\right) \tone_{\Omega_{2, m}}+\mathbb{E}\left[e^{-c_f|\tilde{Z}^h_{t_{n+1}}|} f\left(|\tilde{Z}^h_{t_{n+1}}|\right) \tone_{\{|\tilde{Z}^h_{t_{n+1}}-Z^h_{t_n}|<a\}} 1_{\Omega_{2, m}}\right] \\
& \leq \eta_3 \mathbb{E} f\left(|\tilde{Z}^h_{t_{n+1}}|\right) \tone_{\Omega_{2, m}}+2 e^{-c_f \mu_m} \mathbb{E}\left[f\left(|\tilde{Z}^h_{t_{n+1}}|\right) \tone_{\{|\tilde{Z}^h_{t_{n+1}}-Z^h_{t_n}|<a\}} \tone_{\Omega_{2, m}}\right] \\
& \leq\left(\eta_3+2 e^{-c_f \mu_m}\right) \mathbb{E} \left[f\left(|\tilde{Z}^h_{t_{n+1}}|\right) \tone_{\Omega_{2, m}}\right].
\end{aligned}
$$
Clearly, when $h$ is small, we have $\eta_3 \leq e^{-c_f \mu_m}$ as $\eta_3$ is exponentially small.

Combining the diffusion and drift step, one has
$$
\begin{aligned}
\mathbb{E} \left[\tone_{\Omega_{2, m}} f\left(|Z^h_{t_{n+1}}|\right) \right]
& \leq\left(1+12M  \frac{R_f}{f\left(R_f\right)} e^{-c_f \mu_m} h\right) \exp \left(- \frac{c_f}{f\left(R_f\right)} e^{-c_f \mu_m} h\right) \mathbb{E} \left[\tone_{\Omega_{2, m}} f\left(\left|Z^h_{t_n}\right|\right)\right].
\end{aligned}
$$
Taking large $c_f$ ($c_f > 12MR_f = 24MR' = 24M(4 + 16M / m)R$), and summing up all $m$, there exists positive $C$ independent of $h$ and $n$ such that
\begin{equation}
     \mathbb{E}\left[f(|Z^h_{t_{n+1}}|)\textbf{1}_{\Omega_2}\right]\leq e^{-C h} \mathbb{E}\left[f(|Z^h_{t_n}|)\textbf{1}_{\Omega_2}\right].
\end{equation}

\textbf{Case 3. Consider $\Omega_3$.}
In the far-field region, the contraction is obvious. Indeed, by\eqref{eq:fPhiclaim1} in Lemma \ref{lmm:fPhi}, 
\begin{equation}\label{eq:usefphiclaim1}
\begin{aligned}
    &\quad\mathbb{E}\left[f(|Z^h_{t_{n+1}}|)\textbf{1}_{\Omega_3}\right] = \mathbb{E}\left[f(|Z^h_{t_{n+1}}|)\textbf{1}_{\Omega_3}\textbf{1}_{\{|\tilde{Z}_{t_{n+1}}^h| \leq 3R'/2 \}}\right] + \mathbb{E}\left[f(|Z^h_{t_{n+1}}|)\textbf{1}_{\Omega_3}\textbf{1}_{\{|\tilde{Z}_{t_{n+1}}^h| > 3R'/2 \}}\right]\\
    &\leq e^{C_2'h}\mathbb{E}\left[f(|\tilde{Z}^h_{t_{n+1}}|)\textbf{1}_{\Omega_3}\textbf{1}_{\{|\tilde{Z}_{t_{n+1}}^h| \leq 3R'/2 \}}\right] + e^{-C_1'h}\mathbb{E}\left[f(|\tilde{Z}^h_{t_{n+1}}|)\textbf{1}_{\Omega_3}\textbf{1}_{\{|\tilde{Z}_{t_{n+1}}^h| > 3R'/2 \}}\right].
\end{aligned}
\end{equation}
We prove in Lemma \ref{lmm:technical2} that (take $a = R'/2$, $b= 3R'/2$ and $t = t_{n+1}$ therein)
\begin{equation*}  \mathbb{E}\left[f(|\tilde{Z}^h_{t_{n+1}}|)\textbf{1}_{\Omega_3}\textbf{1}_{\{|\tilde{Z}_{t_{n+1}}^h| \leq 3R'/2 \}}\right] \leq \eta_3(h) \mathbb{E}\left[f(|\tilde{Z}^h_{t_{n+1}}|)\textbf{1}_{\Omega_3}\right],
\end{equation*}
with $\lim_{h \rightarrow 0}\eta_3(h) = 0$. Consequently, 
\begin{equation*}
    \mathbb{E}\left[f(|Z^h_{t_{n+1}}|)\textbf{1}_{\Omega_3}\right] \leq \left( e^{-C_1'h} + \left(e^{C_2'h} - e^{-C_1'h} \right) \eta_3(h)\right)\mathbb{E}\left[f(|\tilde{Z}^h_{t_{n+1}}|)\textbf{1}_{\Omega_3}\right].
\end{equation*}
Moreover, by Lemma \ref{lmm:ddtEfZ}, $\mathbb{E}f(|\tilde{Z}^h_t|)$ is non-increasing for $t \in [t_n,t_{n+1}]$, which implies for small $h$ 
\begin{equation*}
    \mathbb{E}\left[f(|Z^h_{t_{n+1}}|)\textbf{1}_{\Omega_3}\right] \leq \left( e^{-C_1'h} + \left(e^{C_2'h} - e^{-C_1'h} \right) \eta_3(h)\right)\mathbb{E}\left[f(|Z^h_{t_n}|)\textbf{1}_{\Omega_3}\right] \leq e^{-C h} \mathbb{E}\left[f(|Z^h_{t_n}|)\textbf{1}_{\Omega_3}\right].
\end{equation*}

Concluding the one-step contraction results in Case 1 -- Case 3, one has
\begin{equation*}
     \mathbb{E}\left[f(|Z^h_{t_{n}}|)\right]
 \le e^{-C nh} \mathbb{E}\left[f(|Z_0|)\right] =e^{-Cnh}W_f(\mu_0,\nu_0),
\end{equation*}
which implies
\begin{equation*}
     W_f (\mu^h_n, \nu^h_n) \leq e^{-Cnh} W_f(\mu_0, \nu_0).
\end{equation*}
Moreover, since $e^{-c_f R_f}r \leq f(r) \leq r$ for all $r \geq 0$, one obtains
\begin{equation*}
    W_1(\mu^h_n, \nu^h_n) \leq C_0e^{-Cnh} W_1(\mu_0, \nu_0),\quad C_0 := e^{c_f R_f}.
\end{equation*}

\end{proof}

\section{Extension to a long-time Wasserstein error bound}\label{sec:longtime}

Recall that we have established a (finite-time) error estimate of iLMC under the relative entropy. However, for a sampling algorithm, researchers tend to show great interest in studying its long-time behavior. In this section, we extend the finite-time relative entropy error bound to a uniform-in-time Wasserstein-1 error bound.
Our derivation relies on the three main facts:
\begin{enumerate}
    \item The already obtained results, including the (finite-time) relative entropy error bound in Theorem \ref{thm:errormain} and the Wasserstein-1 contraction result in Theorem \ref{thm:contraction} above; 
    \item The triangular inequality of $W_1$ distance; 
    \item The initial conditions in Assumption \ref{ass2} can be propagated along the Fokker-Planck equation for $\rho_t$ (see Proposition \ref{prop:propagation} below). 
\end{enumerate}
% (1) the already obtained results, including the (finite-time) relative entropy error bound in Theorem \ref{thm:errormain} and the Wasserstein-1 contraction result in Theorem \ref{thm:contraction} above; (2) the triangular inequality of $W_1$ distance; (3) the initial conditions in Assumption \ref{ass2} can be propagated along the Fokker-Planck equation for $\rho_t$ (see Proposition \ref{prop:propagation} below). 
% % For better understanding, let us give a brief overview of the derivation here:

% Denote $\mathcal{S}^h$ the one-step Markov transition kernel of iLMC \eqref{eq:ilmciteration} with step size $h$, and $\mathcal{S}(t)$ the Markov transition kernel of the overdamped Langevin equation \eqref{eq:overdamped} over time $t$. 

The detailed derivations are given below. For convenience, we have moved the proof for Proposition \ref{prop:propagation} to Appendix \ref{app:fp}, because most of the proof is identical to that of Proposition \ref{prop:nablalog} and derivations in \cite[Appendix A]{li2025ergodicity}.
 
% \tcb{take a rest here}

% By leveraging logarithmic Sobolev inequalities and entropy dissipation estimates, we can control the Wasserstein distance between the distribution of $X_n$ and the target distribution, even when $\nabla U$ is not globally Lipschitz.

% Moreover, combining the relative entropy error bound and the geometric ergodicity derived in Sections ..., we are able to extend the Wasserstein error estimate to a uniform-in-time one. Indeed, ... (same as \cite{li2025ergodicity})

\begin{proposition}\label{prop:propagation}
Suppose Assumptions \ref{ass0}, \ref{ass1}, \ref{ass2} hold.  Then for all $t\ge 0$
\begin{equation}
C_1'\exp(-C_2 |x|^{\ell_1}) \leq \rho_t(x) \leq C_3'\exp(-\gamma U(x)),\quad |\nabla \log \rho_t(x)| \leq C\left(1 + |x|^{\ell_0} \right).
\end{equation}
Here, the coefficients above are independent of $t$.
\end{proposition}

The extension from a relative entropy bound to a $W_1$ bound requires the following mild condition, which is standard for the weighted Csiszar-Kullback-Pinsker inequality \cite{bolley2005weighted, mou2022improved}.
In detail, the weighted Csiszar-Kullback-Pinsker inequality says that if a probability
$\rho$  has the following tail behavior with a positive constant $a_0$
\begin{equation}\label{eq:subgaussian}
    a_0:=2 \inf_{\alpha>0}\left(\frac{1}{2 \alpha}\left(1+\log \int_{\mathbb{R}^d} e^{\alpha|x|^2} d \rho(x)\right)\right)^{\frac{1}{2}}<+\infty,
\end{equation}
then for any probability measure $\rho' \ll \rho$,
\begin{equation}\label{eq:aftertransport}
    W_1(\rho' , \rho) \leq a_0 \sqrt{\mathcal{H}(\rho' \| \rho)}.
\end{equation}
Clearly, in order for \eqref{eq:subgaussian} to hold uniformly with $\rho = \rho_t$ (recall that $\rho_t$ solves \eqref{eq:overdampedFP}), a sufficient condition is that its initial $\rho_0$ is SubGaussian (namely, there exists some $C>0$ such that $\mathbb{P}(|X_0|>a)\leq\exp(-a^2/C^2)$ for all $a \geq 0$). In fact, under some mild assumptions, it is easy to derive the equivalent characterization of the SubGaussian property of $\rho_t$: $\mathbb{E}[\exp(\alpha |X_t|^2)] \leq 2$ for $X_t$ solving the overdamped Langevin equation \eqref{eq:overdamped} and $\alpha > 0$. Clearly, this further means its law $\rho_t$ satisfies \eqref{eq:subgaussian} uniformly. Moreover, Assumption \ref{ass2} already means that $\rho_0$ is subGaussian since $U$ has a quadratic lower bound. We conclude the above result in the following lemma.
\begin{lemma}\label{lmm:trans}
Suppose the assumptions of Theorem \ref{thm:errormain} hold. Then there exists a positive constant $a_0$ independent of $t$ and $h$ such that for all $t \geq 0$,
\begin{equation}
    W_1(\rho_t^h , \rho_t) \leq a_0 \sqrt{\mathcal{H}(\rho^h_t \| \rho_t)}.
\end{equation}
\end{lemma}

% \begin{remark}
% In order for \eqref{eq:subgaussian} to hold, a sufficient condition is that the initial $\rho_0$ is SubGaussian (namely, there exists some $C>0$ such that $\mathbb{P}(|X_0|>a)\leq\exp(-a^2/C^2)$ for all $a \geq 0$). In fact, under some mild assumptions, it is easy to derive the equivalent characterization of the SubGaussian property of $\rho_t$: $\mathbb{E}[\exp(\alpha |X_t|^2)] \leq 2$ for $X_t$ solving the overdamped Langevin equation \eqref{eq:overdamped} and $\alpha > 0$. Clearly, this further means its law $\rho_t$ satisfies \eqref{eq:subgaussian}.

%     % (some remarks on when Assumption \ref{asstail} holds) \tcb{todo}
% \end{remark}

With the preparations above, we are then able to extend the error bound in Theorem \ref{thm:errormain} to a uniform-in-time one by combining the $W_1$ contraction result in Theorem \ref{thm:contraction}.

\begin{theorem}[Uniform-in-time Wasserstein error estimate for iLMC]\label{thm:longtime}
Suppose the assumptions of Theorem \ref{thm:errormain} hold. Then there exists a positive constant $C$ such that
\begin{equation}\label{eq:uniformintimeW1}
    \sup_{s \geq 0} W_1\left( \rho^h_s, \rho_s\right) \leq Ch.
\end{equation}
Consequently, the invariant measures $\pi$, $\pi^h$ satisfies that
\begin{equation}\label{eq:invariant}
    W_1(\pi^h, \pi) \leq Ch.
\end{equation}
\end{theorem}

\begin{proof}
We first establish the uniform-in-time estimate \eqref{eq:uniformintimeW1}.
Denote $\mathcal{S}^h$ the one-step Markov transition kernel of iLMC \eqref{eq:ilmciteration} with step size $h$, and $\mathcal{S}(t)$ the Markov transition kernel of the overdamped Langevin equation \eqref{eq:overdamped} over time $t$. From the local error analysis in relative entropy (Theorem \ref{thm:errormain}), by Lemma \ref{lmm:trans} due to the weighted Csiszar-Kullback-Pinsker inequality \cite{bolley2005weighted},  for any probability measure $\rho$ satisfying Assumption \ref{ass2}, one has
$$
W_1\left(\mathcal{S}(h)^n \rho,\left(\mathcal{S}^h\right)^n \rho\right) \leq a_0 \sqrt{\mathcal{H}\left(\mathcal{S}(h)^n \rho \mid \left(\mathcal{S}^h\right)^n \rho \right)} \leq C(T) h, \quad \text { for all } n h \leq T.
$$
% To get the uniform-in-time estimate, we combine the finite-time estimate with the exponential ergodicity of the numerical dynamics, together with the uniform estimates of the density for the time-continuous dynamics.
By Theorem \ref{thm:contraction}, there exists $T_0>0$, $\gamma' \in (0,1)$ such that for any $n \geq T_0 / h$ and any probability measures $\mu, \nu$, one has
$$
W_1\left(\left(\mathcal{S}^h\right)^n \mu,\left(\mathcal{S}^h\right)^n \nu\right) \leq \gamma'W_1(\mu, \nu),
$$
Now, set $n_0=\left\lceil T_0 / h\right\rceil$ and take $n=k n_0, m=(k-1) n_0$. By Markov property and the triangular inequality, one has 
$$
\begin{aligned}
W_1\left(\mathcal{S}(h)^n \rho_0,\right. & \left.\left(\mathcal{S}^h\right)^n \rho_0\right) \\
& \leq W_1\left(\mathcal{S}(h)^{n-m} \rho_{t_m},\left(\mathcal{S}^h\right)^{n-m} \rho_{t_m}\right)+W_1\left(\left(\mathcal{S}^h\right)^{n-m} \rho_{t_m},\left(\mathcal{S}^h\right)^{n-m} \rho^h_{t_m}\right).
\end{aligned}
$$
Since $(n-m) h=n_0 h \leq T_0+h$, and since $\rho_{t_m}$ has uniform estimates by Proposition \ref{prop:propagation}, the first term on the right-hand side is bounded by $C\left(T_0\right) h$ and the constant $C\left(T_0\right)$ is independent of $k$. Moreover, by Theorem \ref{thm:contraction}, the second term is bounded by
$$
W_1\left(\left(\mathcal{S}^h\right)^{n-m} \rho_{t_m},\left(\mathcal{S}^h\right)^{n-m} \rho^h_{t_m}\right) \leq \gamma^{\prime} W_1\left(\rho_{t_m}, \rho^h_{t_m}\right)
$$
Hence,
$$
W_1\left(\mathcal{S}(h)^{k n_0} \rho_0,\left(\mathcal{S}^h\right)^{k n_0} \rho_0\right) \leq C h+\gamma^{\prime} W_1\left(\mathcal{S}(h)^{(k-1) n_0} \rho_0,\left(\mathcal{S}^h\right)^{(k-1) n_0} \rho_0\right),
$$
where $C$ is independent of $t$ and $h$. By iteration, this then establishes \eqref{eq:uniformintimeW1} for $n=k n_0$. For general $n$, one only needs to apply the finite-time error estimate again, starting from the nearest integer of $k n_0$.

Finally, note that under Assumption \ref{ass0}, $\pi$ satisfies a log-Sobolev inequality and thus it is well known that \cite{bakry2006diffusions, vempala2019rapid, chewi2024analysis} $\mathcal{S}(t)$ is also geometrically ergodic.  Letting $n \rightarrow \infty$ in \eqref{eq:uniformintimeW1}, the second claim \eqref{eq:invariant} for the invariant measures then follows.

\end{proof}

\section{Conclusion}\label{sec:conclusion}
The iLMC method is a robust approach for sampling from complex distributions with non-globally Lipschitz drift terms, where its explicit competitors usually fails. Its implicit structure enables stable behavior, and through continuous-time interpolation, one can derive meaningful estimates and guarantees such as ergodicity and uniform-in-time sampling error bounds. In this paper, we rigorously give a relative entropy error bound for iLMC, where a crucial gradient estimate for the logarithm numerical density is obtained via a sequence of PDE techniques, including Bernstein method for gradient estimate. We also give a novel framework to prove the geometric ergodicity of iLMC under Wasserstein-1 distance. Based on the relative entropy error bound and the Wasserstein ergodicity, we extend the error bound of iLMC to a uniform-in-time one.

We finally discuss some problems related to our results that still remain open. First, the relative entropy bound in Theorem \ref{thm:errormain} is not the most satisfactory, since the coefficient therein has an algebraic dependence on the time $T$. The reason is that under the current assumptions and techniques, the estimate for (upper and lower) bounds of the numerical density $\rho^h$ in Lemma \ref{lmm:logrho} depends on $T$ exponentially. It is temping and quite promising to seek more advanced methods to get rid of this dependency, so as to improve the relative entropy error bound. Second, the contraction result proved in this paper is for Wasserstein-1 distance only, and it is natural to ask whether one can obtain similar results in Wasserstein-p distances ($p \geq 1$). Although most related contraction result using the reflection coupling is limited to the Wasserstein-1 distance \cite{eberle2016reflection, eberle2019couplings, schuh2024global}, there does exist Wasserstein-p contraction results, where as a trade-off of larger $p$, the dependence of the initial is not tight (see \cite[Theorem 2.1]{wang2020exponential} and \cite[Theorem 1.3]{luo2016exponential}). While interesting, extending from Wasserstein-1 to Wasserstein-p is beyond the scope of this paper

% \tcb{todo}

\section*{Acknowledgments}
The work of L. Li was partially supported by the National Key R\&D Program of China, Project Number 2021YFA1002800, NSFC 12371400, and Shanghai Municipal Science and Technology Major Project 2021SHZDZX0102. This material is based in part upon work supported by the National Science Foundation under Grant No. DMS-2424139, while J.-G. L. was in residence at the Simons Laufer Mathematical Sciences Institute in Berkeley, California, during the Fall 2025 semester.

\appendix

\section{Omitted proofs for the relative entropy error estimate}\label{app:nablalogrho}

We first prove the uniform-in-time $p$-th moment bound stated in Proposition \ref{prop:explicitSDE}. Although the moment bound is well-established in literature, we remark that in most existing results, the bound is limited to a finite-time one or an $L^2$ one. To our knowledge, the proof given below is novel under the current assumption (Assumption \ref{ass0}).
\begin{proof}[Proof of claim \eqref{eq:Lpmomentbound} in Proposition \ref{prop:explicitSDE}] 
Fix $p \geq 2$. Denote $\tilde{X}^h_{t_n} = X^h_{t_n} + \sqrt{2}\,\Delta W_n$. It is easy to show that
\begin{equation}\label{eq:heatstability}
    \mathbb{E}\left[|\tilde{X}^h_{t_n} |^p \mid X^h_{t_n}\right] \leq (1 + \delta h) |X^h_{t_n} |^p + C \delta^{-1} h,\quad \forall \delta > 0,
\end{equation}
where $C$ is a positive constant independent of $h$ and $\delta$. Indeed, let $\tilde{X}_t := X^h_{t_n} + \int_{t_n}^t dW$ for $t \in [t_n,t_{n+1}]$. By It\^o's formula, it holds
\begin{equation*}
    \frac{d}{dt}\mathbb{E}\left[|\tilde{X}_t|^p\mid X^h_{t_n}\right] \leq \frac{1}{2}p(p-2+d)\mathbb{E}\left[|\tilde{X}_t|^{p-2}\mid X^h_{t_n}\right], 
\end{equation*}
\eqref{eq:heatstability} then holds due to Young's inequality and Gr\"onwall's inequality.

Now we consider the drift step.
% recall that it is identical to the well-defined one-step JKO scheme:
% \begin{equation*}
%     X^h_{t_{n+1}} = \argmin_{x \in \mathbb{R}^d} \left\{U(x) + \frac{1}{2h}\left|x - \tilde{X}^h_{t_n}\right|^2 \right\},
% \end{equation*}
% which implies $U(X^h_{t_{n+1}}) + \frac{1}{2h}|X^h_{t_{n+1}} - \tilde{X}^h_{t_n}|^2 \leq U(\tilde{X}^h_{t_n})$.111
Recall the definition of the map $\Phi_h$ in \eqref{eq:defphi}. Note that $0 = \Phi^{-1}_h(\Phi_h(0)) = \Phi^{-1}_h(h\nabla U(0))$. Without loss of generality, assume $0 \in \argmin_x U(x)$. Then by Proposition \ref{prop:PhiLip} (recall $R'$ therein), one has
\begin{equation*}
|X^{h}_{t_{n+1}}-0| \leq \left\{
\begin{aligned}
& e^{-\tfrac{m}{4} h}|\tilde{X}_{t_n} - 0|,\quad \text{if} \quad |\tilde{X}_{t_n}| > R',\\
& e^{2M h}R',\quad \text{otherwise}.
\end{aligned}
\right.
\end{equation*}
Hence,
\begin{equation}\label{eq:stabilityofphi}
    |X^h_{t_{n+1}}|^p \leq e^{2Mph} R'^p \vee e^{-\tfrac{m}{4} ph}|\tilde{X}_{t_n}|^p.
\end{equation}
Taking expectation and combining with \eqref{eq:heatstability}, choosing $\delta = \frac{m}{8}ph$, one has
\begin{equation}\label{eq:onestepiteration}
\begin{aligned}
    \mathbb{E}|X^h_{t_{n+1}}|^p &\leq \max\left((1-\frac{m}{4}ph)\left((1 + \delta h) \mathbb{E}|X^h_{t_n} |^p + C \delta^{-1} h\right) , e^{2Mph} R'^p\right) \\
    &\leq \max\left((1 - C'h) \mathbb{E}|X^h_{t_n} |^p +  C''h , C'''\right),
\end{aligned}
\end{equation}
where $C'$, $C''$, $C'''$ are positive constants independent of $h$ and $n$. Then one has by iteration that
\begin{equation*}
    \sup_{n \in \mathbb{N}} \mathbb{E}|X^h_{t_n}|^p < \infty.
\end{equation*}
Furthermore, performing the above estimates again, one knows that \eqref{eq:onestepiteration} still holds if replacing $t_{n+1}$ by $t$ ($\in [t_n,t_{n+1}]$), and $h$ by $t-t_n$. Therefore,
\begin{equation*}
    \sup_{t \geq 0} \mathbb{E}|X^h_{t_n}|^p < \infty.
\end{equation*}
    
\end{proof}

Next, we prove the upper and lower bounds for the numerical density $\rho^h$. The basic idea is to consider the time evolution of $\rho^h / \tilde{q}$ for some function $\tilde{q}: \mathbb{R}_{+} \times \mathbb{R}^d \rightarrow \mathbb{R}$ and then apply the maximal principle.

\begin{proof}[Proof of Lemma \ref{lmm:logrho}]
\quad

\textbf{1. Proof of the lower bound.}

Let $q := \rho^h / \tilde{q}$, where $\tilde{q}(t,x)$ is to be determined. We then derive the time evolution equation for $q$. In fact, since
\begin{equation*}
    \partial_t (q\tilde{q}) = - \partial_{i}(b_h^i q\tilde{q}) + \partial_{ij}(\Lambda_h^{ij}q\tilde{q}),
\end{equation*}
we have
\begin{equation*}
    \partial_t q = \Lambda_h^{ij} \partial_{ij}q - b_h^i \partial_{i}q + 2\partial_{j}(\Lambda_h^{ij}) \partial_{i}q + 2 \Lambda_h^{ij} \frac{\partial_{j}\tilde{q}}{\tilde{q}} \partial_{i}q + Fq,
\end{equation*}
where the function $F$ is defined by
\begin{equation*}
    F := -\frac{\partial_t \tilde{q}}{\tilde{q}} - \partial_{i}(b^i_n) - b_h^i \frac{\partial_{i}\tilde{q}}{\tilde{q}} + (\partial_{ij}\Lambda_h^{ij}) + 2\partial_{j}(\Lambda_h^{ij}) \frac{\partial_{i}\tilde{q}}{\tilde{q}} + \Lambda_h^{ij} \frac{\partial_{ij}\tilde{q}}{\tilde{q}}
\end{equation*}
Denote the equation above by 
\begin{equation*}
    \mathcal{L}q + Fq = 0,
\end{equation*}
Next, we will show that there exists $C_t > 0$ and $\ell_1$, $C_2$ in Assumption \ref{ass2},  for
\begin{equation*}
    \tilde{q}(t,x) = \exp\left(- C_t t - C_2|x|^{\ell_1}\right),
\end{equation*}
one has
\begin{equation*}
    F \geq 0.
\end{equation*}
Indeed, define $\lambda$ the smallest eigenvalue of $a$ and
\begin{equation*}
    \underline{\lambda} := \min_{x \in B(x^{*},1)}\lambda = \left(I + (t-t_n) \max_{x \in B(x^{*},1)}|\nabla^2 U(x)| \right)^{-2}.
\end{equation*}
For the $\tilde{q}$ of the above form, since $\Lambda_h^{ij}(x)$ is positive definite, one has
\begin{equation*}
\begin{aligned}
    F &\geq C_t - \left|\partial_{ij}(\Lambda_h^{ij}) - \partial_{i}(b_h^i) \right| - \left|2\partial_{j}(\Lambda_h^{ij}) - b_h^i \right| \left(C_2\ell_1 |x|^{\ell_1 - 1} \right) + \underline{\lambda}\frac{\partial_{ij}\tilde{q}}{\tilde{q}}\\
    &=C_t - \left|\partial_{ij}(\Lambda_h^{ij}) - \partial_{i}(b_h^i) \right| - \left|2\partial_{j}(\Lambda_h^{ij}) - b_h^i \right| \left(C_2\ell_1 |x|^{\ell_1 - 1} \right)\\
    &\quad+ \underline{\lambda}\left(C_2^2 \ell_1^2|x|^{2\ell_1 - 2} - C_2 \ell_1 (\ell_1 - 1) |x|^{\ell_1 - 2} \right)\\
    &=C_t - \left|\partial_{ij}(\Lambda_h^{ij}) - \partial_{i}(b_h^i) \right| - \left|2\partial_{j}(\Lambda_h^{ij}) - b_h^i \right| \left(C_2\ell_1 |x|^{\ell_1 - 1} \right) \\
    &\quad+ \underline{\lambda} C_2 \ell_1|x|^{\ell_1 - 2}\left(C_2 \ell_1|x|^{\ell_1} - (\ell_1 - 1) \right)
\end{aligned}
\end{equation*}
By Lemma \ref{lmm:abcpoly}, there exists $C>0$, $\ell \geq 1$ such that
\begin{equation*}
    \max\left\{\left|\partial_{ij}(\Lambda_h^{ij}) - \partial_{i}(b_h^i) \right|, \left|2\partial_{j}(\Lambda_h^{ij}) - b_h^i \right| \right\} \leq C(1 + |x|^\ell).
\end{equation*}
Also, by definition of $\Lambda_h$, denoting $\lambda_{\max}$ the largest eigenvalue of $\nabla^2 U(x)$, one has
\begin{equation*}
    \underline{\lambda}^{-1} = (1 + (t-t_n)\lambda_{\max})^2 \leq 2 + 2h^2C(1+|x|^{2\ell}).
\end{equation*}
Then,
\begin{equation*}
\begin{aligned}
    F &\geq C_t - C\left(1 + |x|^\ell\right)\left(1 + C_2 \ell_1 |x|^{\ell_1 - 1} \right)\\
    &\quad+  \underline{\lambda} C_2 \ell_1|x|^{\ell_1 - 2}\left(C_2 \ell_1|x|^{\ell_1} - (\ell_1 - 1) \right)
\end{aligned}
\end{equation*}
When $|x| \leq r_0 := \frac{\ell_1 - 1}{C_2\ell_1}$, recalling that $M(r) = \max_{B(0,r)}|\nabla^2 U(x)|$, then
\begin{equation}
\begin{aligned}
    F &\geq C_t - C(1 + r_0^\ell)(1 + C_2 \ell_1r_0^{\ell_1 - 1}) + 0 -(1-(t-t_n)M(r_0))^{-1}C_2 \ell_1 (\ell_1 - 1)r_0^{\ell_1 - 2}\\
    & =: C_t - A_0.
\end{aligned}
\end{equation}
When $|x| >r_0 $,
\begin{equation}
\begin{aligned}
    F &\ge C_t - C(r_0^{-\ell} + 1)(r_0^{-(\ell_1 - 1)} + 1)|x|^{\ell_1 + \ell - 1}\\
    &\quad+\left(2r_0^{-2\ell} + 2h^2 C(r_0^{-2\ell} + 1) \right)^{-1}C_2 \ell_1 |x|^{\ell_1 - 2 - 2\ell}\left(C_2 \ell_1|x|^{\ell_1} - (\ell_1 - 1) \right)\\
    &\quad =: C_t + A_1|x|^{2\ell_1 - 2\ell - 2} - A_2|x|^{\ell_1 - 2 - 2\ell} - A_3|x|^{\ell_1 + \ell - 1}.
\end{aligned}
\end{equation}
Above, $A_0, A_1, A_2, A_3 \in \mathbb{R}_{+}$. Clearly, since $\ell_1 \geq 3\ell + 2$, and using Young's inequality, one knows that there exists $A_6 > 0$ such that when $|x| > r_0$,
\begin{equation*}
    F \geq C_t - A_6.
\end{equation*}
Hence, choosing large $C_t$ such that $C_t \geq A_0 \vee A_6$ gives
\begin{equation*}
    F \geq 0.
\end{equation*}

Finally, since $\Lambda_h$ is positive definite for all $x$, by maximal principle, the minimum of $q(t,x)$ can only be achieved at $t=0$ (Otherwise, at the maximal point, $0 < \mathcal{L}(q) = -Fq \leq 0$). Clearly, $q|_{t=0} \geq C_1$ by Assumption \ref{ass2}. This then gives the desired lower bound, with the constant $C'_1$ possibly being time-dependent.

\textbf{2. Proof of the upper bound.} 

Similarly as we did when proving the lower bound, define $q := \rho^h / \tilde{q}$, where $\tilde{q}(t,x)$ is to be determined. Then
\begin{equation*}
    \partial_t q = \Lambda_h^{ij} \partial_{ij}q - b_h^i \partial_{i}q + 2\partial_{j}(\Lambda_h^{ij}) \partial_{i}q + 2 \Lambda_h^{ij} \frac{\partial_{j}\tilde{q}}{\tilde{q}} \partial_{i}q + Fq,
\end{equation*}
where the function $F$ is defined by
\begin{equation*}
    F := -\frac{\partial_t \tilde{q}}{\tilde{q}} - \partial_{i}(b^i_n) - b_h^i \frac{\partial_{i}\tilde{q}}{\tilde{q}} + \partial_{ij}(\Lambda_h^{ij}) + 2\partial_{j}(\Lambda_h^{ij}) \frac{\partial_{i}\tilde{q}}{\tilde{q}} + \Lambda_h^{ij} \frac{\partial_{ij}\tilde{q}}{\tilde{q}}.
\end{equation*}
Next, we will show that there exists $C'_t > 0$ and $\gamma$ in Assumption \ref{ass2},  for
\begin{equation*}
    \tilde{q}(t,x) = \exp\left( C'_t t - \gamma U(x)\right),
\end{equation*}
one has
\begin{equation*}
    F \leq 0.
\end{equation*}
In fact,  the leading term comes from $-b_h^i \frac{\tilde{q}_i}{\tilde{q}} + \Lambda_h^{ij} \frac{\tilde{q}_{ij}}{\tilde{q}}$. 
Recall the definition of $b_h$ and $\Lambda_h$ in \eqref{eq:bn}, \eqref{eq:Lambdan}. Then
\begin{equation*}
    -b_h^i \frac{\partial_{i}\tilde{q}}{\tilde{q}} + \Lambda_h^{ij} \frac{\partial_{ij}\tilde{q}}{\tilde{q}} = -\gamma(1-\gamma) \nabla U(x) \cdot\left(I + (s - t_n)\nabla^2 U(x)\right)^{-1} \cdot \nabla U(x) + r(s,x),
\end{equation*}
where
\begin{equation}\label{eq:rsxdef}
\begin{aligned}
    r(s,x)&=(s - t_n) \left(I + (s - t_n)\nabla^2 U(x)\right)^{-1} \left( \nabla^3 U(x) :  \left(I + (s - t_n)\nabla^2 U(x)\right)^{-2}\right) \\
    &\cdot (-\gamma\nabla U(x)) +  \left(I + (s - t_n)\nabla^2 U(x)\right)^{-2}: (\gamma \nabla^2 U(x)).
\end{aligned}
\end{equation}
For the leading term, clearly, for $\gamma \in (0,1)$,
\begin{equation*}
\begin{aligned}
    \quad-\gamma(1-\gamma)\nabla U(x)\cdot \left(I + (s - t_n)\nabla^2 U(x)\right)^{-1}  \cdot \nabla U(x) \leq -\frac{\gamma(1-\gamma)|\nabla U(x)|^2}{1 + (s-t_n) |\nabla^2 U(x)|}.
\end{aligned}
\end{equation*}
% For $r(s,x)$, since $|\nabla^2 U(x)| \leq C(1 + |\nabla U(x)|)$ by Assumption ..., one knows that $(s-t_n)\left(I + (s - t_n)\nabla^2 U(x)\right)^{-1}\nabla U(x)$ is bounded. So the first line in $r(s,x)$ above is bounded by a polynomial of $x$ of order $\ell$ (recall the $\ell$ in Assumption ...). The second line in \eqref{eq:rsxdef} is similar since $|\nabla^3 U(x)| \leq C(1 + |\nabla^2 U(x)|)$ by Assumption .... Therefore, $r(s,x)$ is uniformly bounded by a polynomial of $x$ of order $\ell$. 
The remainder $r(s,x)$ is clearly uniformly bounded by $C(|\nabla U(x)| + 1)$ since $\nabla^k U \lesssim 1 + \nabla^{k-1} U$ for $k=2,3,4$, which is assumed in Assumption \ref{ass1}.
Similarly, for the other terms $- \partial_i(b^i_n) + \partial_{ij}(\Lambda_h^{ij}) + 2\partial_j(\Lambda_h^{ij}) \frac{\partial_i\tilde{q}}{\tilde{q}}$, it is easy to check that they are also upper-bounded by $C(|\nabla U(x)| + 1)$. Consequently, 
\begin{equation*}
\begin{aligned}
    F &\leq -C_t -\frac{\gamma(1-\gamma)|\nabla U(x)|^2}{1 + (s-t_n) |\nabla^2 U(x)|} + C(|\nabla U(x)| + 1)\\
    &\leq -C_t + \frac{-\gamma(1-\gamma)|\nabla U(x)|^2 + C\left(\nabla U(x) + 1\right)\left(1 + (s-t_n)C(|\nabla U(x)| + 1)\right)}{1 + (s-t_n) |\nabla^2 U(x)|}.
\end{aligned}
\end{equation*}
Hence, for small $h$ such that $\gamma(1-\gamma) > C^2 h$ above, there exists $\bar{C}>0$ independent of $s$, $n$, $h$ such that
\begin{equation*}
    F \leq -C_t + \frac{\bar{C}}{1 + (s-t_n) |\nabla^2 U(x)|} \leq -C_t + \bar{C}.
\end{equation*}
To conclude, $F \leq 0$ once we choose $\gamma \in (0,1)$ and $C_t \geq \bar{C}$ above.
% $$\tcb{\nabla\cdot\left(\left(I + (s - t_n)\nabla^2 U\right)^{-1}\nabla U\right) \leq \left(I + (s - t_n)\nabla^2 U(x)\right)^{-1} : \nabla^2 U + C| \nabla U|}  = \Delta U + $$ 
% % \tcb{or}
% % $$-\nabla\cdot\left(\left(I + (s - t_n)\nabla^2 U\right)^{-1}\nabla U\right) = -\Delta U + \tcb{(s-t_n)\nabla\cdot\left(\left(I + (s - t_n)\nabla^2 U\right)^{-1}\nabla^2 U\nabla U\right)}$$
% , using exactly the same argument, one has \tcb{!coefficient before $-x\cdot \nabla U$ seems cannot control that before $+x \cdot \nabla U$}
% \begin{equation*}
%     - (b^i_n)_i \leq -\frac{m}{2}|x|^2 + C.
% \end{equation*}
% % For the other terms $ (\Lambda_h^{ij})_{ij} + 2(\Lambda_h^{ij})_j \frac{\tilde{q}_i}{\tilde{q}} + \Lambda_h^{ij} \frac{\tilde{q}_{ij}}{\tilde{q}}$ is bounded by a polynomial of $x$ of order $\ell_2-1$. Therefore, when $|x| > \tilde{R}$, one has
% \begin{equation*}
%     F \leq -C_t' - \gamma'\ell_2 |x|^{\ell_2} + C
% \end{equation*}
% Since $\ell_2 > \ell$, there exists $\tilde{R}' > \tilde{R}$ such that
% \begin{equation*}
%     F\leq 0,\quad \forall |x| > \tilde{R}'.
% \end{equation*}
% For $|x| \leq \tilde{R}'$, by continuity, $\nabla U(x) \cdot \frac{\nabla \tilde{q}}{\tilde{q}}$ is upper bounded uniformly-in-time. Hence,
% \begin{equation*}
%     F\leq -C_t' + C' \leq 0,\quad \forall |x| \leq \tilde{R}',
% \end{equation*}
% by choosing $C_t' \geq C'$. Finally, since $F \leq 0$ globally, 
The conclusion then holds due to the maximal principle similarly as in the proof of the lower bound.

\end{proof}

Next, we give the detailed proof of Proposition \ref{prop:nablalog}.

\begin{proof}[Proof of Proposition \ref{prop:nablalog}]
Fix $T>0$.  By Lemma \ref{lmm:logrho}, $\rho^h_t$ has an upper bound
\begin{equation*}
    M_0 := \exp(C_3'T).
\end{equation*}
Recall that we define $u$ by Cole-Hopf transformation:
$$u(t,x) := \log \frac{\rho^h_t(x)}{M_0}\leq 0.$$
Without loss of generality, below we assume $M=1$. Then, the simple calculations imply that $u$ satisfies the following Hamilton-Jacobi equation:
\begin{equation*}
    \partial_t u=a:\left(\nabla^2 u+\nabla u \otimes \nabla u\right)+b \cdot \nabla u+c.
\end{equation*}
where
\begin{equation}\label{eq:abc}
    a:= \Lambda_h,\quad b:= -b_h+\nabla \cdot \Lambda_h,\quad c:= -\nabla \cdot b_h+\nabla^2: \Lambda_h.
\end{equation}
Recall the definitions of $b_h$ and $\Lambda_h$ in \eqref{eq:bn}, \eqref{eq:Lambdan}. Lemma \ref{lmm:matrixbound} and Assumptions \ref{ass0} -- \ref{ass1} tell that $a$, $b$, $c$, and their first-order derivatives all have polynomial upper bounds with respect to $x$. Namely, there exists $C$ and $\ell$ such that
\begin{equation}\label{eq:polyboundabc}
    \max \{|a|, |b|, |c|, |\nabla a|, |\nabla b|, |\nabla c| \} \leq C(1 + |x|^\ell).
\end{equation}
We give a detailed derivation of \eqref{eq:polyboundabc} in Lemma \ref{lmm:abcpoly}.
Moreover, from Assumption \ref{ass0}, it is clear that the matrix $a(t,x)$ is globally positive definite, but does not have a uniform lower bound of the eigenvalue. Now, with the above properties, we are able to prove the polynomial upper bound for $\nabla u$ using a Bernstein-type method.

As mentioned in Section \ref{sec:nablalogrho}, we construct
\begin{equation*}
    g := \frac{|\nabla u|^2}{(1-u)^2},
\end{equation*}
and denote the nonnegative operator $\mathcal{A}$ by
% \tcb{notation? ($u_{ij} \rightarrow \partial_{ij}u$)}
\begin{equation*}
    \mathcal{A}(g) := a^{ij}\partial_{ij}g  - \partial_{t}g+ b^i \partial_{i}g - 3a^{ij}\frac{\partial_{j}u}{1-u}\partial_{i}g + 2a^{ij} \partial_{j}u \partial_{i}g.
\end{equation*}
% Then, after straightforward calculations, one has\tcb{take a rest here}
% \begin{equation*}
%     \begin{aligned}
% \tilde{\mathcal{L}}(g) = & 2 a^{i j} \frac{u_{i k} u_{j k}}{(1-u)^2}+2 a^{i j} \frac{u_{i k} u_k u_j}{(1-u)^3}+2 a^{i j} \frac{u_k u_k u_i u_j}{(1-u)^3} \\
% & -2 c \frac{u_k u_k}{(1-u)^3}-2 \frac{u_k\left(a_k^{i j}\left(u_{i j}+u_i u_j\right)+b_k^i u_i+c_k\right)}{(1-u)^2}.
% \end{aligned}
% \end{equation*}

Now, we fix $x^* \in \mathbb{R}^d$ and take a cut-off function $\tilde{\psi}(\cdot)$ defined on $[0,\infty)$ satisfying: (1) $\tilde{\psi}(r) > 0$ for $r \in [0,1)$; (2) $\mathrm{supp} \tilde{\psi} \in [0,1]$; (3) $\tilde{\psi}(r) = 1$ for $r \in [0,\frac{1}{2}]$; and (4) for any $\delta \in (0,1)$, there exists $C_\delta > 0$ such that for any $r \geq 0$,
\begin{equation}\label{eq:cutoffproperty}
    |\tilde{\psi}'| + |\tilde{\psi}''| \leq C_\delta \tilde{\psi}^{\delta}.
\end{equation}
Note that such cut-off function $\tilde{\psi}$ does exist, for instance $\tilde{\psi}(r) \sim \exp \left(-(1-r)^{-2}\right)$ as $ r \rightarrow 1^{-} $, see also \cite[Section 4]{du2024collision}, \cite[Section 2]{li2025ergodicity}. Then, we take $\psi(x - x^*) = \tilde{\psi}(|x-x^*|)$, which is a cut-off function on $\mathbb{R}^d$ that vanishes on $B(x^*, 1)^c$. Clearly,
\begin{equation*}
    \mathcal{A}(\psi g)=\psi \mathcal{A} (g)+g\mathcal{A}(\psi)+2 a^{i j} \partial_{j}\psi \partial_{i}g .
\end{equation*}
By \eqref{eq:cutoffproperty} and the fact that $|a| \leq 2$ (recall Lemma \ref{lmm:matrixbound}), one has
\begin{equation*}
    2 a^{i j} \partial_{j}\psi \partial_{i}g=2 a^{i j} \partial_{j}\psi \psi^{-1}\partial_{i}(\psi g)-2 a^{i j}\left(\partial_{i}\psi \partial_{j}\psi / \psi\right) g \geq 2 a^{i j} \partial_{j}\psi\psi^{-1}\partial_{i}(\psi g)-4C_{1/2} g .
\end{equation*}
Similarly, using Cauchy-Schwarz inequality, one has
\begin{equation*}
\begin{aligned}
g \mathcal{A}( \psi) & =g\left(\left(a^{i j} \partial_{ij}\psi+b^i \partial_{i}\psi\right)-3 a^{i j} \frac{\partial_{j}u}{1-u} \partial_{i}\psi+2 a^{i j} \partial_{j}u \partial_{i}\psi\right) \\
&\geq-6C_{1/2}g - C_{1/2}|b|g - 6|\nabla \psi| \sqrt{1-u} g^{3 / 2}.
\end{aligned}
\end{equation*}
By Young's inequality and \eqref{eq:cutoffproperty}, for any $\epsilon > 0$, 
\begin{equation}\label{eq:epsilontochoose}
    6|\nabla \psi| \sqrt{1-u} g^{3 / 2} \leq \epsilon \psi(1-u) g^2 + 9C^2_{1/2}\epsilon^{-1}g.
\end{equation}
For the term $\mathcal{A}(g)$, recall the crucial estimate in Proposition \ref{prop:Ag}. Recall that $\lambda$ denotes the smallest eigenvalue of $a$. Also recall that
\begin{equation*}
    \underline{\lambda} := \min_{x \in B(x^{*},1)}\lambda = \left(I + (t-t_n) \max_{x \in B(x^{*},1)}|\nabla^2 U(x)| \right)^{-2}.
\end{equation*}
Then,
\begin{equation*}
\begin{aligned}
\psi\mathcal{A} (g)\geq \psi\frac{\underline{\lambda}}{2} \frac{\left|\nabla^2 u\right|^2}{(1-u)^2}+\frac{\underline{\lambda}}{2}(1-u)\psi g^2  - M_1 (1-u)  (g+1),
\end{aligned}
\end{equation*}
where the function $M_1$ is defined by
\begin{equation*}
    M_1 := 2|c| + 2\underline{\lambda}^{-1}|\nabla a|^2 + 2|\nabla b| +  2|\nabla c|.
\end{equation*}
Defining the operator
\begin{equation*}
    \hat{\mathcal{A}}(g) := \mathcal{A}(g) - 2a^{ij}\psi_j \psi^{-1} g_i,
\end{equation*}
and concluding the estimates above (choosing $\epsilon = \underline{\lambda} / 4$ in \eqref{eq:epsilontochoose}), one has
\begin{equation*}
    \hat{\mathcal{A}}(\psi g) \geq \frac{\underline{\lambda}}{4}\psi (1-u) g^2 - M_2  (1-u) (g+1),
\end{equation*}
where the function $M_2$ is defined by
\begin{equation*}
    M_2 := M_1 + 9C^2_{1/2}\underline{\lambda}^{-1} + C_{1/2}|b| + 10C_{1/2}.
\end{equation*}

The desired result then follows by studying when the maximum of the function $\psi g(t,x)$ is achieved:
\begin{itemize}
\item \textbf{Case 1: $\psi g$ attains its maximum at $t = 0$.} Then for any fixed $x^*$,
\begin{multline*}
    g(t,x^*) = \psi g(t,x^*) \leq \max_{x \in B(x^*,1)}\psi g(0,x) \leq \max_{x \in B(x^*,1)} g(0,x)\\
    \leq \max_{x \in B(x^*,1)} |\nabla \log \rho_0(x)|^2 \leq \max_{x \in B(x^*,1)}C(1 + |x|^\ell) \leq C' (1 + |x^*|^\ell).
\end{multline*}
Here the constant $C'$ is independent of $x^*$, $h$, $T$ and $t$.
\item \textbf{Case 2: $\psi g$ attains its maximum in $(0,T] \times int(B(x^*,1))$.} Note that $\hat{\mathcal{A}}$ is a parabolic operator on $[0,T] \times B(x^*,1)$ since $a$ is locally positive definite. Denote the maximum point of $\psi g$ by $(t_1,x_1)$. Then,
\begin{equation*}
    0 \geq \hat{\mathcal{A}}(\psi g)(t_1,x_1) \geq \frac{\underline{\lambda}}{4}\psi (1-u) g^2 - M'  (1-u) (g+1) \Big|_{(t_1,x_1)},
\end{equation*}
which implies
\begin{equation*}
    \psi g(t_1 x_1) \leq 4M_2 \underline{\lambda}^{-1} \frac{g + 1}{g}\Big|_{(t_1,x_1)}.
\end{equation*}
If $g(t_1,x_1) \leq 1$, then $g(t,x^*) = \psi g(t,x^*) \leq \psi g(t_1,x_1) \leq 1$. Otherwise,
\begin{equation*}
     g(t,x^*) = \psi g(t,x^*)\psi g(t_1 x_1) \leq 8M_2 \underline{\lambda}^{-1}\Big|_{(t_1,x_1)}.
\end{equation*}
Clearly, by Assumption \ref{ass1}, $\underline{\lambda}^{-1}$ is upper bounded by some polynomial of $x^*$. Combining this with \eqref{eq:polyboundabc}, and since $|x_1 - x^*| \leq 1$, one has
\begin{equation*}
     |g(t,x^*)| \leq C' (1 + |x^*|^\ell).
\end{equation*}
Here, $C'$ is independent of $x^*$, $h$, $t$ and $T$.
\end{itemize}

Recalling that $\log M_0 \lesssim T$, one then has \eqref{eq:nablalogandlog}.
Moreover, by Lemma \ref{lmm:logrho}, \eqref{eq:nablalogrhopoly} holds.

\end{proof}

During the proof of Proposition \ref{prop:nablalog}, we also require the following result:
\begin{lemma}\label{lmm:abcpoly}
Recall the definitions of functions $a:\mathbb{R}^d \times \mathbb{R}_{+} \rightarrow \mathbb{R}^{d\times d}$, $b:\mathbb{R}^d \times \mathbb{R}_{+} \rightarrow \mathbb{R}^{d}$, $c:\mathbb{R}^d \times \mathbb{R}_{+} \rightarrow \mathbb{R}$ defined in \eqref{eq:abc}. Suppose Assumptions \ref{ass0}, \ref{ass1} hold. There exists $\ell \geq 1$ defined in Assumption \ref{ass1} and $C>0$ independent of $s$, $h$  such that for any $x \in \mathbb{R}^d$
\begin{equation}\label{eq:maxabc}
    \max \{|a|, |b|, |c|, |\nabla a|, |\nabla b|, |\nabla c| \} \leq C(1 + |x|^\ell).
\end{equation}
\end{lemma}

\begin{proof}
Recall the definitions,
\begin{equation*}
    a= \Lambda_h,\quad b= -b_h+\nabla \cdot \Lambda_h,\quad c= -\nabla \cdot b_h+\nabla^2: \Lambda_h,
\end{equation*}
where 
\begin{multline*}
    b_h(s, x) :=  -\left(I + (s - t_n)\nabla^2 U(x)\right)^{-1} \nabla U(x) \\
          -  (s - t_n) \left(I + (s - t_n)\nabla^2 U(x)\right)^{-1} \left( \nabla^3 U(x) :  \left(I + (s - t_n)\nabla^2 U(x)\right)^{-2}\right),
\end{multline*}
and 
\begin{equation*}
    \Lambda_h(s,x) := \left(I + (s - t_n)\nabla^2 U(x)\right)^{-2}.
\end{equation*}
Clearly, under Assumption \ref{ass0}, by Lemma \ref{lmm:matrixbound}, for $h < \log 2 / (2M)$, $|\Lambda_h| \leq 2$, and $|\nabla \Lambda_h| \vee |\nabla^2 \Lambda_h| \leq C$ since $\nabla^k U \leq C(|\nabla^{k-1} U + 1)$ for $k =2,3,4$ as we assumed in Assumption \ref{ass1}. Similarly, since $\nabla U$ and $\nabla^2 U$ has polynomial upper bounds, one has $|b_h| \vee |\nabla b| \leq C(1 + |x|^\ell)$ for the $\ell$ in Assumption \ref{ass1}. Note that the positive constant $C$ above is independent of $s$, $n$ and $h$. The claim \eqref{eq:maxabc} then follows.

\end{proof}

\section{Technical Lemmas used in Section \ref{sec:ergodicity}}\label{app:ergodicitylemma}

% (similar to Lemmas ... in [Li, Liu, Wang, 2024] and in \cite{li2025ergodicity})

We first prove the subGaussian property stated in Lemma \ref{lmm:subgaussian}. We also refer the readers to \cite[Lemma 3.2]{li2024geometric} for a similar proof. Recall that the process $\zeta_t$ is defined by 
\begin{equation}\label{eq:zeta_def}
    \zeta_t := \int_{{t_n} \wedge \tau}^{t \wedge \tau} \frac{(\tilde{Z}^h_s)^{\otimes 2}}{|\tilde{Z}^h_s|^2} \cdot dW_s, \quad t\in[t_n,t_{n+1}].
\end{equation}

% \tcb{(copied from [Li, Liu, Wang, 2024], to rewrite the notations...)} \tcb{take a rest here}

\begin{proof}[Proof of Lemma \ref{lmm:subgaussian}]
Fix $t_n \leq s \leq t \leq t_{n+1}$. We prove the subgaussian property via the well-known $\psi_2$-condition \cite{vershynin2018high}: there exists $\alpha > 0$ such that
\begin{equation}\label{eq:psi2}
    \mathbb{E}\left[e^{\alpha|\theta|^2}\mid \mathcal{F}_{t_n}\right] \leq 2,
\end{equation}
where we denote $\theta := \zeta_t - \zeta_s$ and $\mathcal{F}_{t_n}$ the $\sigma$-algebra generated by $(X^h_s, Y^h_s, s\leq t_n)$. Clearly, $\zeta_t$ is a martingale by optional stopping theorem \cite{durrett2018stochastic}, and its quadratic variation satisfies
$\langle \zeta_t\rangle\le h$. Then it holds by the Burkholder-Davis-Gundy (BDG) inequality that for $\alpha > 0$,
\begin{multline}\label{eq:taylorbdg}
    \mathbb{E}\left[e^{\alpha |\theta_t|^2}\mid \mathcal{F}_{t_n}\right] = 1 + \sum_{p=1}^{+\infty} \frac{1}{p!} \alpha^p \mathbb{E}\left[|\theta|^{2p}\mid \mathcal{F}_{t_n}\right]\\
    \leq 1 + \sum_{p=1}^{+\infty} \frac{1}{p!}\alpha^p C_{2p} \mathbb{E}\left[\langle \theta \rangle_{t_{n+1}}^p\mid \mathcal{F}_{t_n}\right] \leq 1 + \sum_{p=1}^{+\infty} \frac{1}{p!} C_{2p} \left(h \alpha\right)^p,
\end{multline}
where $C_{2p}$  is a positive constant satisfying:
\begin{equation}\label{claim:Cq}
    C_{2p} \le (C \sqrt{2p})^{2p},
\end{equation}
and $C$ is a universal positive constant.  Combining \eqref{eq:taylorbdg} and \eqref{claim:Cq}, one has
\begin{equation*}
    \mathbb{E}\left[e^{\alpha |\theta_t^{\tau_j}|^2}\Big| \mathcal{F}_{t_n}\right] \leq 1 + C\sum_{p=1}^{+\infty} \frac{p^p}{p!}\left(2h \alpha \right)^p.
\end{equation*}
Clearly, $\frac{p^p}{p!} \leq e^p p^{-\frac{1}{2}} \leq e^p$, which can be derived from an intermediate result in the proof of Stirling's formula \cite{rudin1976principles}: $\log p! > \left(p+\frac{1}{2}\right)\log p - p$. Therefore,
\begin{equation*}
    \mathbb{E}\left[e^{\alpha |\theta_t^{\tau_j}|^2}\Big| \mathcal{F}_{t_n}\right] \leq 1+ C \sum_{p=1}^{+\infty} \left(2eh \alpha\right)^p = 1 + C\frac{2eh \alpha}{1 - 2eh \alpha} =  2,
\end{equation*}
by choosing $\alpha = \frac{1}{2e(1+C)h }=: \bar{c} h^{-1}$. Therefore, the $\psi_2$ condition \eqref{eq:psi2} holds.
Finally, using Chernoff's bound \cite{vershynin2018high}, for any $a>0$, it holds that
\begin{equation}\label{eq:chernoff}
    \mathbb{P}\left(|\theta_t^{\tau_j}| > a\Big| \mathcal{F}_{t_n}\right) \leq \mathbb{E}\left[e^{\alpha |\theta_t^{\tau_j}|^2}\Big| \mathcal{F}_{t_n}\right] / e^{\alpha a^2} \leq 2 e^{-\bar{c}h^{-1} a^2}.
\end{equation}
\end{proof}

\begin{lemma}\label{lmm:technical1}
Let $0<a<b \leq R_f$ and $ C a^2 / h>4 \log 8$. Fix $t \in [t_n,t_{n+1}]$.  Define the events
\begin{equation*}
    A := \left\{\left|Z^h_{t_n}\right| \leq a \right\},
\end{equation*}
\begin{equation*}
    A' := \left\{\exists s \in\left[t_n, t\right],\left|\tilde{Z}^h_s\right|=b\right\}.
\end{equation*}
Then, one has
\begin{equation}\label{eq:app1claim1}
    \mathbb{E}\left[\textbf{1}_A \textbf{1}_{A'} \textbf{1}_{t<\tau}\right] \leq \eta_1(a, b, h) \mathbb{E}\left[\textbf{1}_A \textbf{1}_{t<\tau}\right],
\end{equation}
where
\begin{equation*}
    \eta_1(a,b,h) :=4 \exp \left(-\frac{C (b-a)^2}{h}\right)
\end{equation*}
Moreover, 
\begin{equation}\label{eq:app1claim2}
    \mathbb{E}\left[\textbf{1}_A\left|\tilde{Z}^h_t\right| \textbf{1}_{\{\left|\tilde{Z}^h_t\right| \geq b\}}\right] \leq \eta_2(a, b, h) \mathbb{E}\left[\textbf{1}_A\left|\tilde{Z}^h_t\right|\right],
\end{equation}
where 
\begin{equation*}
    \eta_2(a,b,h):=\frac{6}{a}\left[b+\frac{h}{C (b-a)}\right] \exp \left(-\frac{C (b-a)^2}{2 h}\right)
\end{equation*}
\end{lemma}

\begin{proof}

Define
$$E:=\left\{\exists s \in\left[t_n, t\right],\left|Z_s\right|=a\right\} \cap A, \quad B:=\{t<\tau\}.$$

We first prove \eqref{eq:app1claim1}. Clearly, the event $A'$ must be contained in $E$. In what follows, we will actually show that
$$
\mathbb{E}\left[\tone_{A \cap E} \tone_{A'} \tone_{\{t<\tau\}}\right] \leq \eta_1(a, b, h) \mathbb{E}\left[\tone_{A \cap E} \tone_{\{t<\tau\}}\right] .
$$
Our main idea is that when $E$ happens, the probability for $\{t<\tau\}$ is large. In fact,
$$
\mathbb{P}(A \cap E, A', t<\tau)=\mathbb{P}(A \cap E, t<\tau) \mathbb{P}(A' \mid A \cap E, t<\tau).
$$
For the latter,
$$
\mathbb{P}(F \mid A \cap E, t<\tau) \leq \frac{\mathbb{P}(A', A \cap E)}{\mathbb{P}(A \cap E)-\mathbb{P}(A \cap E, \tau \leq t)}
$$
Meanwhile,
$$
\mathbb{P}(A, E, t \geq \tau)=\mathbb{P}(A, E) \int_{t_n}^t \mathbb{P}\left(t \geq \tau \mid | \tilde{Z}^h_s |=a, A\right) \nu_{A, E}(d s),
$$
where $\nu_{A, E}(\cdot)$ is the conditional law for the first hitting time of $a$ for $|\tilde{Z}^h|$ with $\int_{t_n}^t \nu_{A, E}(d s)=$ 1. Clearly, by Lemma \ref{lmm:subgaussian},
$$
\mathbb{P}\left(t \geq \tau \mid | \tilde{Z}^h_s |=a, A\right) \leq \mathbb{P}\left(\sup _{s \leq t^{\prime} \leq t \wedge \tau} 2 \sqrt{2 }\left|\zeta_{t^{\prime}}-\zeta_s\right| \geq a \mid | \tilde{Z}^h_s |a, A\right) \leq 2 \exp \left(-\frac{C  a^2}{h}\right)<\frac{2}{8^4}.
$$
Hence, 
$$
\mathbb{P}(A' \mid A \cap E, t<\tau) \leq 2 \mathbb{P}(A' \mid A \cap E) \leq 4 \exp \left(-\frac{C (b-a)^2}{h}\right)=: \eta_1.
$$

The proof of \eqref{eq:app1claim2} uses the similar idea. Note that $\left\{\left|\tilde{Z}^h_t\right|>b\right\}$ must be contained in $E$. We will then actually show that
$$
\mathbb{E}\left[\tone_{A \cap E}\left|\tilde{Z}^h_t\right| \tone_{\left|\tilde{Z}^h_t\right| \geq b}\right] \leq \eta_2(a, b, h) \mathbb{E}\left[\tone_{A \cap E}\left|\tilde{Z}^h_t\right|\right].
$$
Clearly, the followings hold:
$$
\begin{aligned}
& \mathbb{E}\left[\tone_{A \cap E}\left|\tilde{Z}^h_t\right| \tone_{\left|\tilde{Z}^h_t\right| \geq b}\right]=b \mathbb{P}\left(\left|\tilde{Z}^h_t\right| \geq b, A, E\right)+\int_b^{\infty} \mathbb{P}\left(\left|\tilde{Z}^h_t\right| \geq r, A, E\right) d r, \\
& \mathbb{E}\left[\tone_{A \cap E}\left|\tilde{Z}^h_t\right|\right]=\int_0^{\infty} \mathbb{P}\left(\left|\tilde{Z}^h_t\right| \geq r, A, E\right) d r.
\end{aligned}
$$
Hence, it suffices to show that
$$
b \mathbb{P}\left(\left|\tilde{Z}^h_t\right| \geq b \mid A, E\right)+\int_b^{\infty} \mathbb{P}\left(\left|\tilde{Z}^h_t\right| \geq r \mid A, E\right) d r \leq \eta_2 \int_0^{\infty} \mathbb{P}\left(\left|\tilde{Z}^h_t\right| \geq r \mid A, E\right) d r.
$$
Intuitively, $\mathbb{P}\left(\left|\tilde{Z}^h_t\right| \geq r \mid A, E\right)$ is small for $r \geq b$, and is almost $1$ if $r \leq a / 2$. In detail,
$$
\mathbb{P}\left(\left|\tilde{Z}^h_t\right| \geq r \mid A, E\right)=\int_{t_n}^t \mathbb{P}\left(\left|\tilde{Z}^h_t\right| \geq r \mid | \tilde{Z}^h_s |=a, A, E\right) \nu_{A, E}(d s).
$$
when $r \geq b$, 
$$
\mathbb{P}\left(\left|\tilde{Z}^h_t\right| \geq r \mid E, A\right) \leq \sup _s \mathbb{P}\left(\sup _{s \leq t^{\prime} \leq t \wedge \tau} 2 \sqrt{2 }\left|\zeta_{t^{\prime}}-\zeta_s\right| \geq r-a \| \tilde{Z}^h_s \mid=a, A\right) \leq 2 \exp \left(-C (r-a)^2 / h\right).
$$
Similarly, for $r \leq a / 2$, one has $\mathbb{P}\left(\left|\tilde{Z}^h_t\right|>r \mid A, E, \tone_{\{t<\tau\}}\right)=1-\mathbb{P}\left(\left|\tilde{Z}^h_t\right| \leq r \mid A, E, \tone_{\{t<\tau\}}\right)$. Similar as above, one has
$$
\mathbb{P}\left(\left|\tilde{Z}^h_t\right| \leq r \mid A, E, \tone_{\{t<\tau\}}\right) \leq \frac{1}{1-2 \exp \left(-C  a^2 / 4 h\right)} \mathbb{P}\left(\left|\tilde{Z}^h_t\right| \leq r \mid A, E\right)
$$
while
$$
\mathbb{P}\left(\left|\tilde{Z}^h_t\right| \leq r \mid A, E\right) \leq 2 \exp \left(-C  a^2 /(4 h)\right).
$$
Therefore, now it suffices to let the following hold:
$$
\left[2 b+\frac{2 h}{C (b-a)}\right] \exp \left(-\frac{C }{2 h}(b-a)^2\right) \leq \eta_2(a, b, h) \frac{a}{2} \frac{1-4 \exp \left(-C  a^2 /(4 h)\right)}{1-2 \exp \left(-C  a^2 /(4 h)\right)}.
$$
Clearly, the above holds if one chooses
$$
\eta_2(a, b, h)=\frac{2}{a} \frac{1-2 / 8}{1-4 / 8}\left[2 b+\frac{2 h}{C (b-a)}\right] \exp \left(-\frac{C }{2 h}(b-a)^2\right).
$$

\end{proof}

\begin{lemma}\label{lmm:technical2}
Let $a>0$ satisfy $2 \exp \left(-C  a^2 / h\right)<1$. Fix some $b>0$ and $t \in [t_n,t_{n+1}]$. Define the events
\begin{equation*}
    B := \{2 a<\left|Z^h_{t_n}\right| \leq b \},
\end{equation*}
\begin{equation*}
    B' := \left\{\exists s \in\left[t_n, t\right],\left|\tilde{Z}^h_s-Z^h_{t_n}\right|=a\right\}.
\end{equation*}
Then, one has
\begin{equation}\label{eq:app2claim1}
    \mathbb{E}\left[\textbf{1}_B \textbf{1}_{B'} f\left(\left|\tilde{Z}^h_t\right|\right)\right] \leq \eta_3(2 a, b, h) \mathbb{E}\left[\textbf{1}_B f\left(\left|\tilde{Z}^h_t\right|\right)\right],
\end{equation}
where
\begin{equation*}
    \eta_3(2 a, b, h):=\frac{2\left(2 b+\eta_2(b, 2 b, h)\right) \exp \left(-C  a^2 / h\right)}{f(a)\left(1-2 \exp \left(-C  a^2 / h\right)\right)}
\end{equation*}
\end{lemma}

The proof of Lemma \ref{lmm:technical2} is almost the same as the proof of Lemma \ref{lmm:technical1}. The detailed estimates needed include:

\begin{enumerate}
    \item The left-hand side of \eqref{eq:app2claim1} is controlled by
$$
\mathbb{E}\left[\tone_B \tone_{B'}\left|\tilde{Z}^h_t\right|\right] \leq\left[2 b+\eta_2(b, 2 b, h)\right] \mathbb{P}(B, B') \leq 2\left[2 b+\eta_2(b, 2 b, h)\right] \exp \left(-C  a^2 / h\right) \mathbb{P}(B).
$$
\item The expectation on the right-hand side of \eqref{eq:app2claim1} is bounded below by
$$
\mathbb{E}\left[\tone_B f\left(\left|\tilde{Z}^h_t\right|\right)\right] \geq f(a) \mathbb{P}(B)\left[1-\mathbb{P}\left(\left|\tilde{Z}^h_t\right| \leq a \mid B\right)\right] \geq f(a) \mathbb{P}(B)\left(1-2 \exp \left(-C  a^2 / h\right)\right).
$$
\end{enumerate}
We omit the details here.

\section{Some estimates for the Fokker-Planck equation}\label{app:fp}

In what follows, we establish several bounds for the solution of the Fokker-Planck equation \eqref{eq:overdampedFP} corresponding to the overdamped Langevin equation \eqref{eq:overdamped}. In other words, we prove Proposition \ref{prop:propagation} which states that the initial conditions in Assumption \ref{ass2} can be uniformly propagated through the PDE \eqref{eq:overdampedFP}.

The proof for the uniform-in-time upper and lower bounds of $\rho_t$ shares the similar idea as in \cite[Appendix A]{li2025ergodicity}. The other part ($|\nabla \log \rho_t(x)| \leq \mathcal{P}(x) |\log \rho_t(x)|$, $\mathcal{P}(x)$ is a polynomial of $x$) follows the same derivation as that of \eqref{eq:nablalogandlog} in Proposition \ref{prop:nablalog} (recall that the estimate \eqref{eq:nablalogandlog} is already uniform-in-time).

\begin{proof}[Proof of Proposition \ref{prop:propagation}]
Consider
$$
q_t(x):=\rho_t(x) / e^{- U(x)}
$$
Clearly, $q_t$ satisfies a backward Kolmogorov equation
$$
\partial_t q=-\nabla U \cdot \nabla q+\Delta q,
$$
amd it is also well-know that (see for instance \cite[Section 2.2]{li2020large}, \cite[Section 3.2]{li2022sharp})
\begin{equation}\label{eq:feynmankac}
q_t(x)=\mathbb{E} \left[ q_0\left(X_t(x)\right)\right],
\end{equation}
where $X_t(x)$ is the stochastic trajectory of
$$
d X=-\nabla U(X) d t+\sqrt{2} \, d W,\quad X_0 = x.
$$
% Recall the claim that
% $$
% B_1 \exp \left(-C|x|^p\right) \leq q_t(x) \leq B_2 \exp (\beta(1-\delta) U(x)) .
% $$
Assumption \ref{ass2} indicates that (note that $U(x)$ is larger than some quadratic function under Assumption \ref{ass0}):
\begin{equation}\label{eq:q0bound}
C_1^{\prime} \exp \left(-C_2^{\prime}|x|^p\right) \leq q_0(x) \leq C_3\exp \left(\left(1-\gamma\right) U(x)\right) .
\end{equation}

\textbf{Upper bound:}

We apply It\^o's formula to $e^{(1-\gamma) U}$ and obtain that
$$
\frac{d}{d t} \mathbb{E} e^{(1-\gamma) U(X)}=\mathbb{E}\left\{e^{(1-\gamma) U(X)}\left[-(1-\gamma) \gamma|\nabla U(X)|^2+(1-\gamma) \Delta U(X)\right]\right\}.
$$
Since $U$ is strongly convex in the far field by Assumption \ref{ass0},
$$
-(1-\gamma) \gamma|\nabla U(X)|^2+(1-\gamma) \Delta U(X) \leq 0 + \left(- \tilde{C}\tone_{\{|X| > R\}}+C \tone_{\{|X| \leq R\}}\right) .
$$
for $C, \tilde{C} > 0 $. Hence,
\begin{equation}\label{eq:combine3}
\frac{d}{d t} \mathbb{E} e^{(1-\gamma) U(X)} \leq-\tilde{C} \mathbb{E} e^{(1-\gamma) U(X)}+C
\end{equation}
By Gr\"onwall's inequality and combining \eqref{eq:feynmankac}, \eqref{eq:q0bound} and \eqref{eq:combine3}, one has
$$
\sup _{t \geq 0} q_t(x) \leq e^{(1-\gamma) U(x)} e^{-C' t}+C.
$$
Therefore, since $U \geq 0$ by Assumption \ref{ass0}, one has
\begin{equation*}
    \rho_t(x) \leq C_3e^{-\gamma U(x)} e^{-C' t}+Ce^{-U(x)} \leq C_3'e^{-\gamma U(x)}.
\end{equation*}

\textbf{Lower bound:}

We claim that there exists $L>0$ such that for all $x$ and $t$,
\begin{equation}\label{eq:finalclaim}
\mathbb{P}\left(\left|X_t(x)\right| \leq L(|x|+1)\right) \geq 1 / 2 .
\end{equation}

% \tcb{to delete:***************************************************}

% In fact, by Markov's inequality as well as \eqref{eq:combine3}, 
% $$
% \begin{aligned}
% &\quad\mathbb{P}(|X| \geq  L(|x|+1))  \leq \mathbb{E}\left[e^{(1-\gamma)U(X)}\right] / \inf _{y \in \mathbb{R}^d \backslash B(0, L(1+|x|))} e^{(1-\gamma) U(y)}\\
% &\leq\left(e^{(1-\gamma) U(x)} e^{-C't}+C\right) / \inf _{y \in \mathbb{R}^d \backslash B(0, L(1+|x|))} e^{(1-\gamma) U(y)} \\
% &\leq \exp \left((1-\gamma) \inf_{y \in \mathbb{R}^d \backslash B(0, L(1+|x|))}\left(U(x) -U(y)\right)\right) e^{-C' t}+C \exp \left((1-\gamma)\left(-\lambda_1(L+L| x|)^2\right)\right)\\
% &\tcr{\leq} \exp \left((1-\gamma)\left(\lambda_2|x|^2-\lambda_1(L+L| x|)^2\right)\right) e^{-C' t}+C \exp \left((1-\gamma)\left(-\lambda_1(L+L| x|)^2\right)\right)
% \end{aligned}
% $$

% \tcb{the above seems not true if $\nabla U$ is nonLipschitz}

% \tcb{rewrite:}

% \tcb{***************************************************************}

In fact, by Markov's inequality,
\begin{equation*}
\begin{aligned}
\mathbb{P}(|X_t| \geq  L(|x|+1))  \leq \mathbb{E}|X_t|^2 / (L+L|x|)^2
\end{aligned}
\end{equation*}
By It\^o's formula and the far-field convexity of $U$ in Assumption \ref{ass0}, one has
\begin{equation*}
    \frac{d}{dt}\mathbb{E}|X_t|^2 = 1 - \mathbb{E}\left[2X_t \cdot \nabla U(X_t)\right] \leq C - C' \mathbb{E}|X_t|^2\tone_{\{|X_t| > \tilde{R} \}},
\end{equation*}
which implies
\begin{equation*}
    \mathbb{E}|X_t|^2\tone_{\{|X_t| > \tilde{R} \}} \leq \mathbb{E}|X_t|^2 \leq |x|^2 + \int_0^t \left(C - C' \mathbb{E}|X_{t'}|^2\tone_{\{|X_{t'}| > \tilde{R} \}} \right) dt'.
\end{equation*}
Consequently, 
\begin{equation*}
    \sup_{t \geq 0}\mathbb{E}|X_t|^2 \leq C(|x|^2 + 1).
\end{equation*}
Choosing $L > 2C$, one has
\begin{equation*}
    \mathbb{P}(|X_t| \geq  L(|x|+1))  \leq 1/2.
\end{equation*}
Finally, by \eqref{eq:finalclaim} and \eqref{eq:q0bound}, ono has
$$
\inf_t q_t(x) \geq \mathbb{E} C_1' \exp \left(-C_2'\left|X_t(x)\right|^p\right) \geq C_1'' \exp \left(-C_2'' |x|^p\right).
$$

\end{proof}

\bibliographystyle{plain}
\bibliography{main}

\end{document}